\documentclass[11pt]{article}

\usepackage{amsthm}
\usepackage{amsmath}
\usepackage{bbm}
\RequirePackage{amssymb}
\usepackage{verbatim}
\usepackage{graphicx}

\theoremstyle{plain}
\newtheorem{theorem}{Theorem}
\newtheorem{lemma}{Lemma}
\newtheorem{prop}{Proposition}

\theoremstyle{remark}
\newtheorem{remark}{Remark}
\usepackage{cite}

\newcommand{\overbar}[1]{\mkern 4mu\overline{\mkern-4mu#1\mkern-4mu}\mkern 4mu}
\newcommand{\Sup}{M/M/n\text{-Sup}(d)}
\newcommand{\R}{\mathbb{R}}
\newcommand{\Rbar}{\overbar{\R}}
\newcommand{\E}{\mathbb{E}}
\newcommand{\N}{\mathbb{N}}
\newcommand{\vn}[1]{\left|\left|#1\right|\right|}
\newcommand{\st}{\text{ s.t. }}
\newcommand{\bo}[1]{O\left(#1\right)}

\newcommand{\PP}{\mathbb{P}}
\newcommand{\defeq}{\,{\buildrel \triangle \over =}\,}
\newcommand{\1}{\mathbbm{1}}

\newcommand{\bounded}[1]{\mathcal{Z}^{#1}}
\newcommand{\region}[1]{Z^{\alpha}_{#1}}

\title{Supermarket Queueing System in the Heavy Traffic Regime.\\ 
Short Queue Dynamics}
\author{Patrick Eschenfeldt and David Gamarnik}

\begin{document}

\maketitle

\begin{abstract}
We consider a queueing system with $n$ parallel queues operating according to
the so-called ``supermarket model'' in which arriving customers join the
shortest of $d$ randomly selected queues. Assuming rate $n\lambda_{n}$ Poisson
arrivals and rate $1$ exponentially distributed service times, we consider this
model in the heavy traffic regime, described by $\lambda_{n}\uparrow 1$ as
$n\to\infty$. We give a simple expectation argument establishing that majority
of queues have steady state length at least $\log_d(1-\lambda_{n})^{-1} - O(1)$
with probability approaching one as $n\rightarrow\infty$, implying the same for
the steady state delay of a typical customer.

Our main result concerns the detailed behavior of queues with length 
\emph{smaller}
than $\log_d(1-\lambda_{n})^{-1}-O(1)$. Assuming $\lambda_{n}$ converges to $1$
at rate at most $\sqrt{n}$, we show that the dynamics of such queues
does not follow a diffusion process, as is typical for queueing systems in
heavy traffic, but is described instead by a deterministic infinite system of
linear differential equations, after an appropriate rescaling.
The unique fixed point solution of this system is shown explicitly to be of the
form $\pi_{1}(d^{i}-1)/(d-1), i\ge 1$, which we conjecture describes the
steady state behavior of the queue lengths after the same rescaling.
Our result is obtained by combination of several technical ideas including 
establishing the existence and uniqueness of an associated
infinite dimensional system of non-linear integral equations and adopting an
appropriate stopped process as an intermediate step. 
\end{abstract}

\section{Introduction.}\label{sec:intro}
In this paper we consider the so-called supermarket model in the heavy traffic
regime. The supermarket model
is a parallel server queueing system consisting of $n$ identical servers which
process jobs at rate $1$ Poisson process.
The jobs arrive into the system according to a Poisson process with rate
$n\lambda_{n}$ where $\lambda_{n}$ is assumed to be
strictly smaller than unity for stability. A positive integer parameter $d$ is
fixed. Each arriving customer chooses
$d$ servers uniformly at random and selects a server with the smallest number
of jobs in the corresponding queue, ties
broken uniformly at random. The queue within each server is processed according
to the First-In-First-Out rule. 
We denote this system by $\Sup$.

The foundational work on this model was done by Dobrushin, Karpelevich and
Vvedenskaya~\cite{vvedenskaya} and Mitzenmacher~\cite{mitzenmacher},
who independently showed that when $\lambda_{n}=\lambda<1$ is a fixed constant
and $d\ge 2$, the steady state probability that the customer
encounters a queue with length at least $t$ (and hence experiences the delay at
least $t$ in expectation), 
is of the form $\lambda^{d^{t}}$. 
Namely it is doubly exponential in $t$. This is in sharp contrast with the case
$d=1$, where each server behaves as an $M/M/1$ system
with load $\lambda$ and hence the steady state delay has the exponential tail
of the form $\lambda^{t}$. This phenomena has its static
counterpart in the form of so-called Balls-Into-Bins model. In this model $n$
balls are thrown sequentially into $n$ bins where for each ball $d$ bins
are chosen uniformly at random and the bin with the smallest number of balls is
chosen. It is well known that for this model the largest
bin has $O(\log\log n)$ balls when $d\ge 2$ as opposed to $O(\log n)$ balls
when $d=1$. This known as ''Power-of-Two'' phenomena.

The development in~\cite{vvedenskaya} and \cite{mitzenmacher}
is based on the fluid limit approximations for the infinite dimensional
process, where each coordinate corresponds to the fraction of servers
with at least $i$ jobs. By taking $n$ to infinity, it is shown that the
limiting system can be described by a deterministic infinite system 
of differential equations, which have a unique and simple to describe fixed
point satisfying doubly exponential decay rate.
Some of the subsequent work that has been performed on the supermarket model
and its
variations can be found in
\cite{mukherjee2015universality, mukherjee2016efficient, bramson-2010, bramson-2012, bramson-2013, graham, luczak-mcdiarmid, luczak-norris, Mitzenmacher1999, vvedenskaya-97, li, dieker-suk}.

In this paper we consider the supermarket model in the heavy traffic regime
described by having the arrival rate parameter
$\lambda_{n}\uparrow 1$.
The work of Brightwell and Luczak~\cite{luczak} 
considers the model which is the closest to the one considered in this paper.
They also assume that $\lambda_{n}\uparrow 1$, but
at the same time they assume that the parameter $d$ diverges to infinity as
well. In our setting $d$ remains constant as is the case for the
classical supermarket model. More precisely, we assume that as $n$ increases,
$d$ is fixed, but
$\lambda_{n}=1-\beta/\eta_{n}$ where $\beta>0$ is fixed and
$\lim_{n}\eta_{n}=\infty$.
Our goal is conducting the performance analysis of the system both at the
process level and in steady state. Unfortunately,
the fluid limit approach of~\cite{vvedenskaya} and \cite{mitzenmacher} is
rendered useless since in this case the corresponding
fluid limit  trivializes to a system of differential equations describing the
critical system corresponding to $\lambda_{n}=\lambda=1$.
At the same time, however, certain educated guesses can be inferred from the
case when $\lambda<1$ is constant, namely the classical setting.
From the $\lambda^{d^{i}}$ tail behavior which describes
the fraction of servers with at least $i$ customers in steady state, it can be
inferred that when $\lambda\uparrow 1$,
if $i^{*}=o\left(\log_{d}{1\over 1-\lambda}\right)$ then
\begin{align*}
\lambda^{d^{i^{*}}}=\lambda^{o\left({1\over 1-\lambda}\right)},
\end{align*}
which approaches unity as $\lambda\uparrow 1$. Namely, the fraction of servers
with at most $i^{*}$ customers becomes negligible. Of course this does
not apply rigorously to our heavy traffic regime as it amounts to first taking
the limit in $n$ and only then taking the limit in $\lambda$,
whereas in the heavy traffic regime this is done simultaneously. Nevertheless,
our main results confirm this behavior both at the process level
and in the steady state regime. In terms of our notation for $\lambda_{n}$, we 
show that when $\omega_{n}$ is an arbitrary sequence
diverging to infinity and
\begin{align}\label{eq:i-n-star}
i^{*}_{n}=\log_{d}{1\over
1-\lambda_{n}}-\omega_{n}=\log_{d}\eta_{n}-\log_{d}\beta-\omega_{n},
\end{align}
(note that the term $\log_{d}\beta$ is subsumed by $\omega_{n}$)
the fraction of queues with at most $i^{*}_{n}$ customers is $o(1)$ with
probability approaching unity as $n$
increases.
The intuitive explanation for this is as above:
\begin{align*}
\lambda_{n}^{d^{i^{*}_{n}}}&=\left(1-\beta/\eta_{n}\right)^{d^{-\omega_n}
\over 1-\lambda_{n}} \\
&=\left(1-\beta/\eta_{n}\right)^{\eta_{n}\over\beta d^{\omega_{n}}} \\
&\approx e^{-{1\over d^{\omega_{n}}}} \\
&\rightarrow 1,
\end{align*}
as $n\rightarrow\infty$.

We now describe our results and our approach at some level of detail. First we
give a very simple expectation based argument
showing that in steady state the expected fraction of servers with at least $i$
customers is at least 
$1 - {\beta \over \eta_n}{d^i -1\over d-1}$. Plugging here the value for
$i_{n}^{*}$ given in (\ref{eq:i-n-star}) 
the expression becomes $1-1/d^{\omega_{n}}\rightarrow 1$,
as $n\rightarrow\infty$, confirming the claimed behavior in steady state. This
immediately implies that the steady state
delay experienced by a typical customer is at least $i^{*}_{n}$ with
probability approaching $1$ as $n$ diverges.
This result is formally stated in Theorem~\ref{thm:steady-state-expectation}.

Our main result concerns the detailed process level behavior of queues, with
the eye towards queues of 
length at most $i_{n}^{*}$. As is customary in the heavy traffic
theory, the first step is applying an appropriate rescaling step, 
and thus for each $i\le i_{n}^{*}$, letting $S^{n}_{i}(t)$ denote the fraction
of servers with at most $i$
jobs at time $t$, we introduce a rescaled process $T_{i}^{n}(t)\triangleq
\eta_{n}(1-S^{n}_{i}(t))$. We prove that the sequence
$T^{n}=\left(T_{i}^{n}(t), t\ge 0, i\ge 1\right)$ converges weakly to some
deterministic limiting process $T=\left(T_{i}(t), t\ge 0,i\ge 1\right)$,
provided that the system starts in a state where $T_{i}^{n}(0)$ has a
non-trivial limit as $n\rightarrow\infty$, and provided
$\eta_{n}$ grows at most order $\sqrt{n}$.
We show that the process
$T(t)$ is the unique solution to a deterministic infinite system of linear
differential equations (given by (\ref{eq:result-i}) in the body of the paper).

This is the main technical result of the paper. This result is perhaps somewhat
surprising since  the processes arising as heavy traffic limits
of queueing systems is usually a diffusion, and not a deterministic process as
it is in our case.
We further show that the unique fixed point of the process $T(t)=(T_{i}(t),
i\ge 0)$ is of the
form 
$\pi_{i}=\pi_{1}(d^{i}-1)/(d-1)$, consistent with our result regarding the
lower bound on steady state expectation of $S_{i}^{n}$ discussed above.
We also show that this fixed point is an attraction point of the process $T(t)$
and the convergence occurs exponentially fast.

Our main result regarding the weak convergence of the rescaled process $T^{n}$
to $T$ is obtained by employing several technical steps. The first
step is writing the term
$\left(S_{i-1}^{n}\right)^{d}-\left(S_{i}^{n}\right)^{d}=(1-T_{i-1}^{n}/\eta_{n
})^{d}-(1-T_{i}^{n}/\eta_{n})^{d}$ 
(which corresponds to the likelihood that the arriving job increases the
fraction
of servers with at least $i$ jobs), as a sum of a linear function
$dT_{i}^{n}/\eta_{n}-dT_{i-1}^{n}/\eta_{n}$ plus the correction
term $d g^{\eta_{n}}/\eta_{n}$, where
$g^{\eta_{n}}$ is the appropriate correction function, and then showing that
this correction has a smaller order of
magnitude provided $i\le i_{n}^{*}$. 
Then, we prove the existence, uniqueness and continuity property of the
stochastic integral equation governing the behavior of the rescaled 
queue length counting process
$\left(T_{i}^{n}, t\ge 0, i\ge 1\right)$, up to an appropriately chosen
stopping time intended to prevent $T_{i}^{n}$ from ``growing too much''.
The stopping time utilized in this theorem is similar to the one employed by
the authors in a different paper~\cite{eschenfeldt}. Finally,
we apply the martingale method by splitting the underlying stochastic processes
into one part which is a martingale and the compensating
part which has a non-trivial drift. It is then shown that the martingale part
is zero in the limit as $n\rightarrow\infty$,
thanks to the nature of the underlying rescaling.

The remainder of the paper is laid out as follows: Section \ref{sec:setup} 
defines the model and states our main results. 
In Section \ref{sec:steady-state} we prove our results regarding the steady
state regime and prove results 
regarding the properties of the limiting deterministic process $T(t)$.  
In Section
\ref{sec:integral} we will prove Theorem \ref{thm:const-integral} regarding the
existence, uniqueness and the continuity 
property of an infinite dimensional stochastic integral equation system
governing the behavior of the rescaled process $T^{n}(t)$.
In Section \ref{sec:martingales} we construct a representation
of the system as a combination of martingales and integral terms. Section
\ref{sec:martingale-convergence} will establish that these martingales converge
to zero. This section will also include the conclusion of
the proof of Theorem \ref{thm:const-result}. Open questions and conjectures are
discussed in Section~\ref{sec:Conclusions}.

We close this section with some notational conventions. 
We use $\Rightarrow$ to denote weak convergence.
$\R (\R_{+})$ denotes the set of (non-negative) real values. 
$\R_{\ge 1}$ denotes the set of real values greater than or equal to one.
$\Rbar_+ = \R_+ \cup \{\infty\}$ denotes the extended non-negative real line.
$\Rbar_{\ge 1} = \R_{\ge 1} \cup \{\infty\}$ denotes the extended non-negative
real line excluding reals strictly less than $1$.
We equip $\Rbar_{\ge 1}$ with the order topology, in which neighborhoods
of $\infty$ are those sets which contain a subset of the form $\{x > a\}$ for
some $a \in \R$.
Let $\R^\infty$ be the space of sequences $x = (x_0,x_1,x_2,\ldots)$. For $x
\in \R^\infty$ and $\rho > 1$ we define the norm
\[\vn{x}_\rho \defeq \sum_{i \ge 0}\rho^{-i}\left|x_i\right|,\]
where $\vn{x}_\rho = \infty$ is a possibility, and we define the subspace
\[\R^{\infty,\rho} \defeq 
    \left\{x \in \R^\infty \st \vn{x}_{\rho} < \infty\right\}.\] 
Note that while $\vn{\cdot}_\rho$ does not induce the standard product topology on $\R^\infty$,
we will use $\R^\infty$ and $\R^{\infty,\rho}$ with the product topology 
unless noted otherwise.
Let $D_t = D([0,t],\R)$ be the space of cadlag functions from $[0,t]$ to $\R$. 
Let $D^\infty_t = D([0,t],\R^\infty)$ be the space of cadlag functions from
$[0,t]$
to $\R^\infty$.
For $x \in D_t$ we denote the uniform norm
\[\vn{x}_t = \sup_{0 \le s \le t} |x(s)|.\]
For $x \in D^\infty_t$ and $\rho > 1$, we define the $\rho$-norm by
\[\vn{x}_{\rho,t} = \sum_{i \ge 0}\rho^{-i}\vn{x_i}_t,\]
where $\vn{x}_{\rho,t} = \infty$ is a possibility and we define the subspace
\[D^{\infty,\rho}_t \defeq 
    \left\{x \in D^\infty_t \st \vn{x}_{\rho,t} < \infty\right\}.\] 
Observe that while $\vn{\cdot}_{\rho,t}$ does not induce the standard product topology on
$D^{\infty}_t$, we will use $D^\infty_t$ and $D^{\infty,\rho}_t$ with 
the product topology unless noted otherwise.
For $\eta \in \R_+$, let
$[-\eta, \eta]^\infty = \{x \in \R^\infty \st \forall \text{ } i \ge 0,
 |x_i| \le \eta\}.$
Let $D^{\eta}_t = D\left([0,t],[-\eta,\eta]^\infty\right)$ 
be the space of cadlag functions from $[0,t]$ to
$[-\eta,\eta]^\infty$. $D^{\eta}_t$ is also 
equipped with the standard product topology. 
For notational convenience, for $\rho > 1$ and $\eta \in \R_+$ 
we define $D^{\eta,\rho}_t \defeq D^{\eta}_t$.

\section{The model and the main result.}\label{sec:setup}
We consider the supermarket model with $n$ exponential rate one servers each
with their own queue, and Poisson arrivals with rate $\lambda_n n$ for
$\lambda_n < 1$ such that $\lambda_n \uparrow 1$. Specifically, we assume
that there exists some sequence $\eta_n$ and constant $\beta > 0$ such that
$\eta_n \to \infty$ as $n \to \infty$ and
\begin{equation}\label{eq:lambda-rate}
    \lim_{n \to \infty} \eta_n(1-\lambda_n) = \beta.
\end{equation}
We assume $\eta_n \ge 1$ for all $n \ge 1$.
Upon arrival customers select $d \ge 2$ queues uniformly at random 
with replacement and join the shortest of
these queues, with ties broken uniformly at random.

Let $0 \le S_i^n(t) \le 1$ be the fraction of queues with at least $i$
customers (including the customer in service) at time $t$. 
Then the probability that
an arriving customer at time $t$ joins a queue of length exactly $i-1$ is 
\[(S_{i-1}^n(t))^d - (S_i^n(t))^d.\]
As a result, because the overall arrival rate is $\lambda_n n$, the
instantaneous rate of  arrivals to queues of length exactly $i-1$ is
\[\lambda_n n \left((S_{i-1}^n(t))^d - (S_i^n(t))^d\right).\]
Note that an arrival to a queue of length $i-1$ increases $S_i^n$ by $1/n$, and
all other types of arrivals leave $S_i^n$ unchanged. Similarly, a departure
from a queue of length $i$ decreases $S_i^n$ by $1/n$ and any other departures
leave $S_i^n$ unchanged.
The instantaneous rate of departures from queues of length exactly $i$
at time $t$ is
\[nS_i^n(t) - nS_{i+1}^n(t).\]
For $i \ge 1$ let $A_i$ and $D_i$ be independent rate-1 Poisson processes.
Then we can represent the processes $S_i^n$ via random time
changes of these Poisson processes. Specifically, we have
\begin{align}
S_0^n(t) &= 1, \label{eq:s-form-0} \\
S_i^n(t) &= S_i^n(0) + {1 \over n}A_i\left(\lambda_n n
\int_0^t\left(\left(S_{i-1}^n(s)\right)^d-\left(S_i^n(s)\right)^d\right
)ds\right) \notag \\
&\quad\quad - {1 \over n}D_i\left(n \int_0^t\left(S_{i}^n(s) -
S_{i+1}^n(s)\right)ds\right), & i \ge 1. \label{eq:s-form-i}
\end{align}
We let
\[T_i^n = \eta_n(1 - S_i^n),\]
where $\eta_n$ is defined by \eqref{eq:lambda-rate}.
Observe that
\[0 \le T_i^n \le \eta_n.\] 
For technical reasons we restrict our choice of $\eta_n$ to those for which
there exists constant $Q \ge 0$ such that
\begin{equation}\label{eq:eta-rate}
    \eta_n \le Q\sqrt{n},
\end{equation}
for all $n \ge 1$.
That is, we assume $\eta_n = \bo{\sqrt{n}}$. 
We further define $\eta_\infty = \infty$.
We will prove in Theorem
\ref{thm:const-result} that under appropriate conditions,
 $T^n = (T_0^n,T_1^n,\ldots)$  weakly converges to
the solution to a certain integral equation, which we first prove in Theorem
\ref{thm:const-integral} has a unique solution.

For every $\eta \in \R_{\ge 1}$ and $x \in \R$, we let
\begin{align}
g^\eta(x) &\defeq {\eta \over d}\left(1-{x\over\eta}\right)^d - {\eta
\over d} + x
\label{eq:g^n}\\
&= {1 \over d} \sum_{l = 2}^d \binom{d}{l}(-1)^l{x^l \over \eta^{l-1}}, \notag
\end{align}
and also let $g^\infty = 0$ for $\eta = \infty$.

Given $b \in \R^{\infty,\rho}$, $y \in D^{\infty,\rho}_t$, 
$\lambda \in \R$, and $\eta \in \Rbar_{\ge 1}$,
consider the following system of integral equations:
\begin{align}
T_0(t) &= 0,  \label{eq:integral-0}\\
T_i(t) &=  b_i + y_i - \lambda d\int_0^t\left(T_i(s) - T_{i-1}(s) -
g^\eta(T_i(s)) +
g^\eta(T_{i-1}(s)) \right)ds \label{eq:integral-i} \\
&\quad\quad + \int_0^t\left(T_{i+1}(s) - T_i(s)\right)ds, \quad\quad i \ge 1.
\notag
\end{align}
An important special case of this system, which will appear as the limiting
system in Theorem \ref{thm:const-result} below, is when we set $\eta = \infty$
and $\lambda = 1$, and $y = 0$. In this case $b_i = T_i(0)$ for $i \ge 1$.
This system is as follows:
\begin{align}
T_0(t) &= 0, \label{eq:result-0}\\
T_i(t) &= T_i(0) - d\int_0^t\left(T_i(s) - T_{i-1}(s)\right)ds \notag\\
&\quad\quad+
\int_0^t\left(T_{i+1}(s) - T_i(s)\right)ds, & i \ge 1 \label{eq:result-i}.
\end{align}
For any $\eta \in \Rbar_{\ge 1}$, $0 < \alpha < 1/2$, and $\rho > 1$, 
let $i^* = {\alpha \over 2} \log_\rho \eta$ where 
$\log_\rho \infty = \infty$. 
In fact, for our purposes, $\alpha / 2$ can be replaced by any positive number 
strictly smaller than $\alpha$.
Define a stopping time $t^* \in \Rbar_+$ as follows:
\begin{align}
t^* \defeq \inf\bigg\{ s \ge 0 : 
&\exists\text{ } i \st 1 \le i \le i^* \text{ and } |T _i(s)| \ge
\eta^{\alpha} \notag
\\
&\text{ or } \label{eq:stopping-time-full} \\
&\exists\text{ } i \st i > i^* \text{ and } |T_i(s)| \ge \eta + 1 
 \notag  \bigg\}.
\notag
\end{align}
We also define a subset of $\R^{\infty,\rho}$ related to this stopping time.
 Let
\begin{align*}
\bounded{\eta} \defeq \bigg\{x \in \R^{\infty,\rho} \st &\forall \text{ } 1
\le i \le
i^*, \quad |x_i| \le \eta^{\alpha} \\
&\text{ and }  \\
&\forall \text{ } i > i^*, \quad |x_i| \le \eta + 1\bigg\}.
\end{align*}
Finally, we define a subset of the product space 
$\R^{\infty,\rho} \times D^{\infty,\rho}_t \times \R \times \Rbar_{\ge 1}$
equipped with the product topology
 which will allow us to limit our attention to
certain parameter values. Specifically, let
\begin{align*}
\region{K} \defeq \bigg\{(b,y,\lambda,\eta) \in \R^{\infty,\rho} \times 
D^{\infty,\rho}_t \times \R \times \Rbar_{\ge 1} \st 
&b + y(0) \in \bounded{\eta}, \\
\vn{b}_{\rho} &\le K, \vn{y}_{\rho,t} \le K\bigg\}.
\end{align*}
Observe that for $(b,y,\lambda,\eta)$ which are not in $\region{K}$ 
for any $K > 0$, we have $t^* = 0$, and that $\region{K}$ is a closed subset
of 
$\R^{\infty,\rho} \times D^{\infty,\rho}_t \times \R \times \Rbar_{\ge 1}$.

Our first result shows that
\eqref{eq:integral-0}-\eqref{eq:integral-i} has a unique solution
on the interval $[0,t^*]$ and that it defines a map which satisfies a certain 
continuity property.

\begin{theorem}\label{thm:const-integral}
For $t \in [0,t^*]$, the system \eqref{eq:integral-0}-\eqref{eq:integral-i}
has a unique solution $T = (T_i, i \ge 0) \in D^{\infty,\rho}_t$. 
For any $t \ge 0$, defining
\begin{align}
\hat{T}(t) = 
    \begin{cases}
        T(t) & t < t^* \\ T(t^*) & t \ge t^*,
    \end{cases} \label{eq:integral-stopped}
\end{align}
we obtain a function $f : \R^{\infty,\rho} \times D^{\infty,\rho}_t 
\times \R \times \Rbar_{\ge 1} \to D^{\infty,\rho}_t$ 
mapping $(b,y,\lambda,\eta)$ to $\hat{T} = f(b,y,\lambda,\eta)$. 
Moreover, when the domain is restricted to $\region{K}$ for any $K > 0$ equipped with the product
topology, $f$ is continuous for every $t \ge 0$.
\end{theorem}
\begin{remark}\label{rem:eta=infty}
    Note that if $\eta = \infty$ then $t^* = \infty$ and thus $T = \hat{T}$.
Further note that for $\eta < \infty$, the definition of
$t^*$ implies that in fact either $\hat{T} \in D^{\eta+1,\rho}_t$ or $t^* =
0$. In the latter case $\hat{T} = b + y(0)$ is a constant function.
\end{remark}
We prove Theorem \ref{thm:const-integral} in Section \ref{sec:integral}.
We now turn to our main result.
\begin{theorem}\label{thm:const-result}
Suppose $\lambda_n$ satisfies \eqref{eq:lambda-rate} for a sequence $\eta_n$
satisfying \eqref{eq:eta-rate} for some $Q > 0$. Suppose there exists $\rho > 1$ such that
\begin{equation}\label{eq:const-starting-state}
	T^n(0) \Rightarrow T(0) \quad \text{ in } \R^{\infty,\rho} \text{ as } n
\to \infty,
\end{equation}
for some random variable $T(0) \in \R^{\infty,\rho}$. 
Furthermore, suppose 
\begin{equation}\label{eq:starting-expectation-bounded}
    \limsup_{n \to \infty}\E\left[\vn{T^n(0)}_{\rho}\right] < \infty,
\end{equation}
and there exists $0 < \alpha < 1/2$
such that for all sufficiently large $n$, almost surely
\begin{equation}\label{eq:start-inside}
    T_i^n(0) \le \eta_n^{\alpha}, \quad 0 \le i \le i^*.
\end{equation}
Then for any $t \ge 0$,
\[ T^n \Rightarrow T \quad \text{ in } D^{\infty,\rho}_t \text{ as } n \to \infty,\]
where $T$ is the unique solution of the system
\eqref{eq:result-0}-\eqref{eq:result-i}.
\end{theorem}
The motivation for the initial condition assumptions
\eqref{eq:starting-expectation-bounded} and \eqref{eq:start-inside} is as
follows:
as we will see below, we expect that in steady state the limiting system $T$
grows like $T_i = d^i$, so condition
\eqref{eq:starting-expectation-bounded} can be considered as requiring $T^n(0)$
to be consistent with this behavior, with $\rho > d$. Condition
\eqref{eq:start-inside} is similar, as $d^{i^*} \approx
\eta_n^{\alpha / 2}$.
We prove Theorem \ref{thm:const-result} in Section
\ref{sec:martingale-convergence}.

Now we turn to results about the solution $T$ which appears in
Theorem~\ref{thm:const-result}.
\begin{theorem}\label{thm:fixed-point}
    Consider the system of integral equations given by
\eqref{eq:result-0}-\eqref{eq:result-i}.
This system has a fixed point $\pi = (\pi_i, i \ge 0)$ given by
\begin{equation}\label{eq:fixed-point}
    \pi_i = \pi_1{d^i - 1 \over d -1}.
\end{equation}
This fixed point is unique up to the constant $\pi_1$.
\end{theorem}
We also show that this fixed point is attractive:
\begin{theorem}\label{thm:Lyap}
Let $\Phi(t) = \sum_{i \ge 1}d^{-i/2}|T_i(t) - \pi_i|$. If $\Phi(0) < \infty$ 
then $\Phi$ converges exponentially fast to zero. 
Specifically, $\Phi(t) \le \Phi(0) e^{-\left(\sqrt{d}-1\right)^2 t}$ for all $t
\ge 0$.
\end{theorem}

Finally we consider the system in steady state. Via elementary arguments we  
prove a bound on the expectation of the fraction of short queues in steady
state. For $i \ge 0$, let $S_i^n(\infty)$ be the fraction of queues with length
at least $i$ in steady state. For the statement below, we assume $\lambda_n$ is
given by \eqref{eq:lambda-rate} and $\eta_n \to \infty$ is arbitrary. In
particular, the assumption \eqref{eq:eta-rate} is no longer needed.
\begin{theorem}\label{thm:steady-state-expectation}
For $i \ge 0$ we have
\begin{equation}\label{eq:steady-state-expectation}
\E S_i^n(\infty) \ge 1 - \left(1 - \lambda_n\right){d^i -1\over d-1}.
\end{equation}
As a result, for any sequence $\omega_n$ which diverges to infinity as $n
\to \infty$, the fraction of queues with length at least $\log_d \eta_n -
\omega_n$ approaches unity with probability approaching one as $n \to
\infty$. This further implies that a customer arriving in steady state
experiences a delay of at least $\log_d \eta_n - \omega_n$ with
probability approaching one as $n \to \infty$.
\end{theorem}
We prove Theorems \ref{thm:fixed-point}-\ref{thm:steady-state-expectation} in
Section \ref{sec:steady-state}.

Note that for any sequence $\omega_n$ which diverges to infinity as $n
\to \infty$, the right hand side of \eqref{eq:steady-state-expectation}
diverges to negative infinity for any $i \ge \log_d \eta_n + \omega_n$. Because
\eqref{eq:steady-state-expectation} is only a lower bound this is not useful,
but it does suggest that the behavior of $S_i^n$ is best examined
for values of $i$ near $\log_d\eta_n$.

\section{The model in steady state.}\label{sec:steady-state}

\begin{proof}[Proof of Theorem \ref{thm:fixed-point}]
We set the derivative of $T_i$ to zero and
and introduce the notation $\pi = (\pi_0,\pi_1,\pi_2,\ldots)$ for the
desired fixed point. This leads to the recurrence
\[\pi_{i+1} = (d+1)\pi_i - d\pi_{i-1}, \quad\quad i \ge 1,\]
which is solved by
\[\pi_i = {1 \over d-1}\left( d \pi_0 - \pi_1 + d^i(\pi_1 -
\pi_0)\right), \quad\quad i \ge 0.\]
By \eqref{eq:result-0} we have $T_0(t) = 0$ for all $t \ge 0$ and thus $\pi_0
= 0$, so this fixed point reduces to
\[\pi_i = \pi_1{d^i - 1 \over d - 1}, \quad\quad i \ge 0.\]
\end{proof}
Next we prove Theorem \ref{thm:Lyap}:
\begin{proof}[Proof of Theorem \ref{thm:Lyap}]
   Define $\epsilon_i(t) = T_i(t) - \pi_i$. We have
\begin{align*}
{d\epsilon_i \over dt} &= d(\epsilon_{i-1} +\pi_{i-1}) -
(d+1)(\epsilon_i+\pi_i) + (\epsilon_{i+1}+\pi_{i+1}) \\
&= d \epsilon_{i-1} - (d+1)\epsilon_i + \epsilon_{i+1}.
\end{align*}

Temporarily assume $\epsilon_i \ne 0$ for all $i$ so the derivative $d\Phi
\over dt$ is well defined. After providing a basic argument under this
assumption we will explain how to remove it.
Now we have
\begin{align*}
{d\Phi \over dt} &= \sum_{i : \epsilon_i > 0}d^{-i/2}\left(d 
\epsilon_{i-1} - (d+1)\epsilon_i + \epsilon_{i+1}\right) \\
&\quad\quad - \sum_{i : \epsilon_i < 0}d^{-i/2}\left(d \epsilon_{i-1} -
(d+1)\epsilon_i + \epsilon_{i+1}\right).
\end{align*}
Let us now consider the terms involving $\epsilon_i$. We will first consider $i
\ge 2$. 
There are several cases,
depending on the signs of $\epsilon_{i-1},\epsilon_i$, and $\epsilon_{i+1}$.
First suppose they are all negative, so the term involving $\epsilon_i$, which
we denote $A_i$, is
\begin{align*}
A_i &= -d^{-(i-1)/2}\epsilon_i + d^{-i/2}(d+1)\epsilon_i - d^{-(i
+1)/2}d\epsilon_i \\
&= d^{-i/2}\epsilon_i\left(-\sqrt{d} + d + 1 - {d \over \sqrt{d}}\right) \\
&= \left(1 - 2\sqrt{d} + d\right)d^{-i/2}\epsilon_i.
\end{align*}
We define
\[\delta = 1 - 2\sqrt{d} + d = \left(\sqrt{d} - 1\right)^2\]
and note that the assumption $d \ge 2$ implies $\delta > 0$. Now note that if
the sign of $\epsilon_{i-1}$ or $\epsilon_{i+1}$ or both
is positive and $\epsilon_i$ remains negative this will simply
change the sign of the appropriate coefficient of $\epsilon_i$ from negative to
positive, decreasing $A_i$. 
Thus for all cases with $\epsilon_i$ negative we have
\[A_i \le \delta d^{-i/2}\epsilon_i.\]
If all three signs are positive, we have
\begin{align*}
A_i &= d^{-(i-1)/2}\epsilon_i - d^{-i/2}(d+1)\epsilon_i + d^{-(i
+1)/2}d\epsilon_i \\
&= d^{-i/2}\epsilon_i\left(\sqrt{d} - d - 1 + {d \over \sqrt{d}}\right) \\
&= -\delta d^{-i/2}\epsilon_i,
\end{align*}
and if $\epsilon_{i-1}$ or $\epsilon_{i+1}$ is negative we still have
\[A_i \le -\delta d^{-i/2}\epsilon_i.\]
Thus for all $i \ge 2$ we have
\[A_i \le -\delta d^{-i/2}|\epsilon_i|.\]
To see that this inequality also holds for $i = 1$ note that for that case
we simply omit the first term of $A_i$.
Using this result for all $i \ge 1$ we conclude
\begin{align*}
{d\Phi \over dt} &= \sum_{i \ge 1} A_i \\
&\le \sum_{i \ge 1}-\delta d^{-i/2}|\epsilon_i| \\
&= -\delta \Phi.
\end{align*}
Thus we have
\[\Phi(t) \le \Phi(0)e^{-\delta t},\]
and conclude that $\Phi(t)$ converges exponentially.

As in Mitzenmacher \cite{mitzenmacher}, because we are interested in the
evolution of the system as time increases,
we can account for the $\epsilon_i = 0$
case by considering upper right-hand derivatives of $\epsilon_i$, defining
\[\left.{d|\epsilon_i| \over dt}\right\vert_{t = t_0} \defeq \lim_{t \to
t_0^+}{|\epsilon_i(t)| \over t-t_0},\]
and similarly for ${d\Phi \over dt}$. Now if $\epsilon_i(t_0) = 0$ we have
$\left.{d|\epsilon_i| \over dt}\right\vert_{t = t_0} \ge 0$, so we can include
the $\epsilon_i = 0$ cases in the above proof with the $\epsilon_i > 0$ case
now also including the case with $\epsilon_i = 0$ and ${d\epsilon_i \over dt}
\ge 0$ and similarly for $\epsilon_i < 0$. 
\end{proof}

\begin{proof}[Proof of Theorem \ref{thm:steady-state-expectation}]
From \eqref{eq:s-form-i} we obtain for $i \ge 1$
\[\E S_i^n(t) = \E S_i^n(0) +
\lambda_n\int_0^t\E\left[\left(S_{i-1}^n(s)\right)^d
- \left(S_i^n(s)\right)^d\right]ds 
- \int_0^t\left(\E S_{i}^n(s) - \E S_{i+1}^n(s)\right)ds.\]
Assuming $(S_i^n(0), i \ge 0)$ has a steady state distribution, the same
applies to $(S_i^n(t), i \ge 0)$, implying $\E S_i^n(0) = \E S_i^n(t)$. Thus
switching to $S_i^n(\infty)$ for steady state version of
$S_i^n(t)$, we obtain
\[0 = \lambda_n\E\left[\left(S_{i-1}^n(\infty)\right)^d
- \left(S_i^n(\infty)\right)^d\right] 
- \left(\E S_{i}^n(\infty) - \E S_{i+1}^n(\infty)\right).\]
Because $0 \le S_i^n(t) \le 1$ for all $i \ge 0$ and $t \ge 0$ and
$S_{i-1}^n(t) \ge S_i^n$ for all $i \ge 1$ and $t \ge 0$, we have the bound
\[\left(S_{i-1}^n(\infty)\right)^d - \left(S_i^n(\infty)\right)^d \le 
d\left(S_{i-1}^n(\infty)-S_i^n(\infty)\right),\]
and thus have
\[0 \le \lambda_n d\left(\E S_{i-1}^n(\infty) - \E S_i^n(\infty)\right) - 
\left(\E S_{i}^n(\infty) - \E S_{i+1}^n(\infty)\right).\]
For $i \ge 0$ define
\[\sigma_i \defeq \E S_i^n(\infty) - \E S_{i+1}^n(\infty),\]
and observe that $\sigma_i$ is the expected number of queues of length exactly
$i$ in steady state. We now obtain the bound
\[\sigma_i \le \lambda_n d \sigma_{i-1}, \quad\quad i \ge 1,\]
which implies
\[\sigma_i \le \sigma_0d^i\left(\lambda_n\right)^i.\]
We have $S_0^n(t) = 1$ and use Little's law to observe $S_1^n(\infty) =
\lambda_n$ resulting in
\[\sigma_i \le (1 - \lambda_n)d^i\lambda_n^i \le (1 - \lambda_n)d^i, \quad i
\ge 1.\]
Now observe
\begin{align*}
    \E S_i^n(\infty) &= \E S_0^n(\infty) - \sum_{j = 0}^{i-1}\left(\E
S_j^n(\infty) - \E S_{j+1}^n(\infty)\right) \\
&= 1 - \sum_{j = 0}^{i-1}\sigma_j \\
&\ge 1 - \left(1 -\lambda_n\right)\sum_{j = 0}^{i-1}  d^j 
= 1 - \left(1-\lambda_n\right) {d^i - 1 \over d - 1}.
\end{align*}
This establishes \eqref{eq:steady-state-expectation}.

Recalling \eqref{eq:lambda-rate},
for $i \le \log_d\eta_n - \omega_n$, we have
\[\E S_i^n(\infty) \ge 1 - \left(1-\lambda_n\right)  {\eta_n d^{-\omega_n} - 1
\over d - 1} \to 1,\]
as $n \to \infty$. Since $S_i^n(\infty) \le 1$, this implies that as $n \to
\infty$, $S_i^n(\infty) \to 1$ in probability.
Namely, the fraction of queues with length at least $\log_d\eta_n - \omega_n$
approaches one in probability.

Finally observe that the probability of an arriving customer in steady state
joining a queue of length at least $i$ is $S_i^n(\infty)^d$. For $i \le
\log_d\eta_n - \omega_n$ because we have $S_i^n(\infty)^d \to 1$ in
probability, a customer
arriving in steady state experiences a delay of at least $\log_d \eta_n -
\omega_n$ with probability approaching one as $n \to \infty$.
\end{proof}

Beyond this elementary bound on the expectation in steady state,
Theorem \ref{thm:const-result} suggests that $T^n(\infty)$ converges to
$\pi$, though formally this is a conjecture because we do not show that the
sequence $T^n(\infty)$ is tight. Establishing this interchange of limits would
be a potential direction for future work on this system. For the remainder of
this section we will suppose the conjecture is true and consider the
implications.

First, treating the fixed point \eqref{eq:fixed-point} as the limit of
$T^n(\infty)$, we have
\[\pi_1 = \lim_{n \to \infty} \eta_n(1-S_1^n(\infty)).\]
We use Little's law to replace $S_1^n(\infty)$ by $\lambda_n$ so we have
\[\pi_1 = \lim_{n \to \infty} \eta_n(1-\lambda_n) = \beta,\]
and therefore the fixed point becomes
\[\pi_i = \beta{d^i - 1 \over d - 1}, \quad\quad i \ge 0.\]

Recall that $T_i(t) = \eta_n\left(1 - S_i(t)\right)$, so 
this fixed point suggests that in steady state the fraction of servers with at
least $i$ jobs can be approximated by
\[S_i^n(\infty) \approx 1 - {\beta \over \eta_n}{d^i-1 \over d-1}\]
when $i = \log_d\eta_n - \omega_n$.

\begin{figure}
    \includegraphics[scale=0.6]{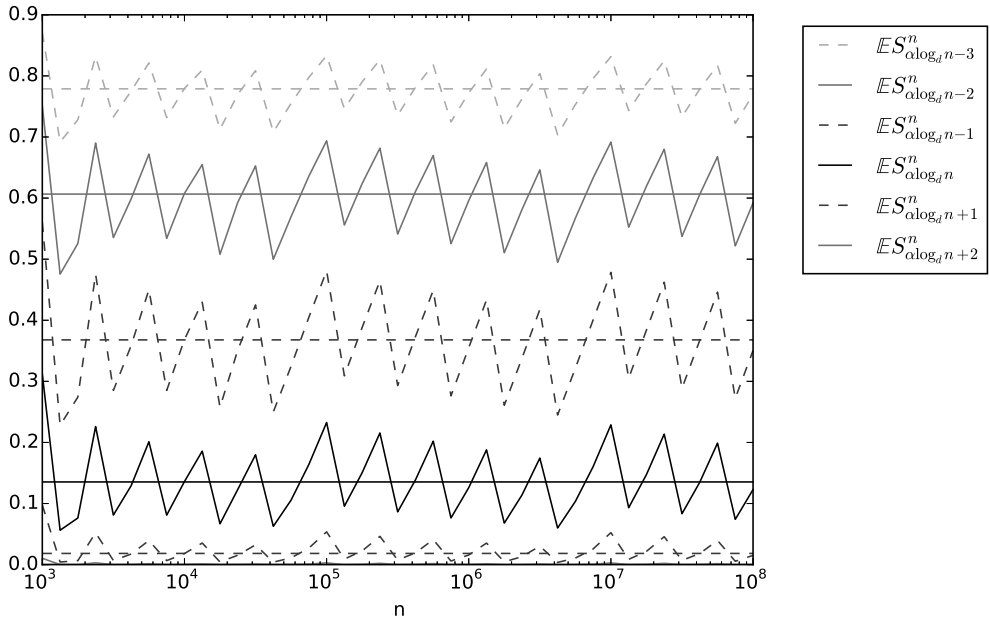}
\caption{Simulated steady state expectation for \Sup\text{ } with $d = 2$ and
$\lambda_n = 1 - \beta n^{-\alpha}$ with $\beta = 2$ and $\alpha = 3/4$. For
each
line the quantity $\alpha \log_d n$ is rounded to the nearest integer. Each
horizontal line indicates $\exp(-\beta d^k / (d-1))$, which is the conjectured
limit of $\E S^n_{\alpha \log_d n + k}$ as $n$ diverges to infinity.}
\label{fig:steady-state}
\end{figure}

We further conjecture that delay times longer than $\log_d \eta_n + \omega_n$
are unlikely.
This is informed by a heuristic analysis of the heavy traffic supermarket model
using the fixed $\lambda < 1$ results proved by Mitzenmacher
\cite{mitzenmacher} and Vvedenskaya, et.~al.~\cite{vvedenskaya}. For fixed
$\lambda < 1$, the system converges to a limiting system which has a unique 
fixed point at
\[\pi_i = \lambda^{d^i-1 \over d-1}.\]
If we simply replace $\lambda$ with $\lambda_n = 1 - {\beta \over \eta_n}$,
then we have
\[\pi_i = \left(1 -{\beta \over \eta_n}\right)^{d^i-1 \over d-1} \to 
\begin{cases} 1 & i \le  \log_d\eta_n - \omega_n \\ 
e^{-{\beta d^k \over d-1}} & i = \log_d \eta_n + k \\
0 & i \ge \log_d \eta_n + \omega_n,   
\end{cases}\]
where $\omega_n$ is any sequence diverging to infinity and $k$ is any constant.
In particular, this heuristic suggests that in steady state all queues will
have length $\log_d \eta_n + \bo{1}$. 

As further evidence for this conjectured behavior, we simulated the system for
a variety of values of $n$ to estimate the expectation in steady state.
Figure~\ref{fig:steady-state} 
shows that the fraction of queues of length at
least $i = \log_d \eta_n + k$ remains approximately constant as $n$ increases.
Furthermore, these simulated steady state expectations appear to vary around
the conjectured limits of $\E S_i^n$ for such $i$, with the variation primarily
introduced by the necessary rounding of $i$ to an integer value.

\section{Integral representation.}\label{sec:integral}
We prove Theorem \ref{thm:const-integral} in this section. 
We will make use of a version of
Gronwall's inequality, which we state now as a lemma (see, e.g., pg. 498 of
\cite{ethier-kurtz}).

\begin{lemma}\label{thm:gronwall}
Suppose that $g : [0,\infty) \to [0,\infty)$ is a function such that
\[0 \le g(t) \le \epsilon + M\int_0^tg(s)ds, \quad 0 \le t \le T,\]
for some positive finite $\epsilon$ and $M$. Then
\[g(t) \le \epsilon e^{Mt}, \quad 0 \le t \le T.\]
\end{lemma}
We begin by establishing two lemmas related to the function $g^\eta$.

\begin{lemma}\label{thm:g^m-converge}
Let $0 < \alpha < 1/2$, and $\rho > 1$, and let
$\eta^n,\eta \in \Rbar_{\ge 1}$ be such that $\eta^n \to \eta$,
$i^*_n = {\alpha \over 2} \log_\rho \eta^n$,
$x_i^n \in D([0,t],[-(\eta^n)^{\alpha}, (\eta^{n})^{\alpha}])$ for $0 \le i \le
i^*_n$ and
$x_i^n \in D([0,t],[-\eta^n-1,\eta^{n}+1])$ for $i > i^*_n$. Then for any 
$i_0 \in \N$, $t \ge 0$,
\[\vn{g^{\eta^n}(x_i^n) - g^{\eta}(x_i^n)}_{t} \to 0, \quad i \le i_0.\]
\end{lemma}

\begin{proof}
First suppose $\eta < \infty$. Then for large enough $n$ we have $\eta_n <
\infty$.
For $i \ge 0$ we have
\begin{align*}
\vn{g^{\eta^n}(x^n_i)-g^{\eta}(x^n_i)}_t &= \sup_{0 \le s \le
t}{1 \over d}\left|\sum_{l = 2}^d\binom{d}{l}(-1)^l x^n_i(s)^l
\left(\left(\eta^{n}\right)^{1-l}-\eta^{1-l}\right)\right| \\
&\le {1 \over d}\sum_{l = 2}^d\binom{d}{l}
\left(\eta^{n}+1\right)^l\left|\left(\eta^{n}\right)^{1-l}-\eta^{1-l}\right|
\defeq C_n.
\end{align*}
Observe that as $\eta^n \to \eta$, $C_n \to 0$, as desired.

Suppose now $\eta = \infty$ and thus $g^\eta = 0$. For $0 \le i \le i^*_n$
we have
\begin{align*}
\vn{g^{\eta^n}(x^n_i)-g^{\eta}(x^n_i)}_t 
&= \vn{g^{\eta^n}(x^n_i)}_t \\
&\le {1 \over d}\sum_{l =2}^d\binom{d}{l}\left(\eta^{n}\right)^{\alpha l} 
\left(\eta^{n}\right)^{1-l} \\
&=  {1 \over d}\sum_{l =
2}^d\binom{d}{l}\left(\eta^{n}\right)^{1-l(1-\alpha)}.
\end{align*}
As $\eta^n \to \infty$ we have $i^*_n \to \infty$ and thus for large enough 
$n$, $i_0 < i^*_n$.
For $2 \le l \le d$ we have
\[\left(\eta^{n}\right)^{1-l(1-\alpha)} \to 0,\]
and thus
\[\vn{g^{\eta^n}(x^n_i)}_t \to 0.\] 
\end{proof}

\begin{lemma}\label{thm:g^m-lipschitz}
For every $\eta \in \R_{\ge 1}$, $g^\eta$ is a Lipschitz
continuous function with constant $4^d$ on $D^{\eta+2}_t$ 
equipped with the topology induced by $\vn{\cdot}_{\rho,t}$.
That is, for $x^1,x^2 \in D^{\eta+2}_t$, and for $i \ge 0$,
\[\vn{g^\eta(x^1_i)- g^\eta(x^2_i)}_{t} 
\le 4^d \vn{x^1_i-x^2_i}_{t}\]
and
\[\vn{g^\eta(x^1)-g^\eta(x^2)}_{\rho,t}
\le 4^d\vn{x^1-x^2}_{\rho,t}.\]
\end{lemma}

\begin{proof}
Consider the restriction of $g^\eta$ onto $[-\eta-2,\eta+2] \to \R$.
This function is differentiable and thus is Lipschitz continuous with constant
\begin{align*}
\sup_{y \in [-\eta-2,\eta+2]}\left|\dot{g}^\eta(y)\right| &=
\sup_{y \in [-\eta-2,\eta+2]}\left|-\left(1-{y \over
\eta}\right)^{d-1} + 1\right| 
\le 4^d.
\end{align*}
Thus for $y^1,y^2 \in [-\eta-2,\eta+2]$
we have $|g^\eta(y^1)-g^\eta(y^2)| \le 4^d|y^1-y^2|$, which further implies
for $x^1,x^2 \in D^{\eta+2}_t$ and for $i \ge 0$ we have
\begin{align*}
\vn{g^\eta(x^1_i)- g^\eta(x^2_i)}_t &= 
\sup_{0 \le s \le t}|g^\eta(x_i^1(s)) - g^\eta(x_i^2(s))| \\
&\le 4^d\sup_{0 \le s \le t}|x_i^1(s)-x_i^2(s)| \\
&\le 4^d\vn{x_i^1-x_i^2}_t. 
\end{align*}
This further implies
\begin{align*}
\vn{g^\eta(x^1)-g^\eta(x^2)}_{\rho,t} &= \sum_{i \ge
0}\rho^{-i}\vn{g^\eta(x^1_i)- g^\eta(x^2_i)}_t \\
&\le \sum_{i \ge 0}\rho^{-i}4^d\vn{x^1_i-x^2_i}_t\\
&= 4^d\vn{x^1-x^2}_{\rho,t}.
\end{align*}
\end{proof}

\begin{proof}[Proof of Theorem \ref{thm:const-integral}: Existence and
uniqueness]
Fix $(b,y,\lambda,\eta) \in \R^{\infty,\rho} \times D^{\infty,\rho}_t
\times \R \times \Rbar_{\ge 1}$.

Suppose first $T(0) = b + y(0) \not\in \bounded{\eta}$. Then $t^* = 0$ and
$\hat{T}(t) = T(0)$ for all $t \ge 0$.

We now suppose $T(0) = b + y(0) \in \bounded{\eta}$, and therefore
for all $1 \le i \le i^*$ we have
$|T_i(0)| \le \eta^{\alpha}$ and for all $i > i^*$ we have $|T_i(0)| \le
\eta + 1$.
We will show existence and uniqueness via a contraction mapping argument,
showing that the map defined by the right hand side of
\eqref{eq:integral-0}-\eqref{eq:integral-i} is a
contraction for small enough $t$. Note that this contraction argument will use the unbounded $\rho$-norm, and uniqueness of the solution with respect to that topology implies uniqueness with respect to the product topology.

We first define the map
$\Gamma:D^{\infty,\rho}_t \to D^\infty_t$, where for $x \in D^{\infty,\rho}_t$, 
\begin{align*}
\Gamma(x)_0(t) &= 0, \\
\Gamma(x)_i(t) &= b_i + y_i(t) - \lambda d\int_0^t\left(x_i(s) - x_{i-1}(s) -
g^\eta(x_i)
+ g^\eta(x_{i-1})\right)ds \\
&\quad\quad + \int_0^t\left(x_{i+1}(s) -
x_i(s)\right)ds, \quad\quad i \ge 1.
\end{align*}
Now let
\begin{align*}
t^* \defeq \inf\bigg\{ s \ge 0 : 
&\exists\text{ } i \st 1 \le i \le i^* \text{ and }
\left|\Gamma(x)_i(s)\right| \ge \eta^{\alpha}  
\\
&\text{ or }  \\
&\exists\text{ } i \st i > i^* \text{ and } \left|\Gamma(x)_i(s)\right| \ge
\eta + 1  \bigg\}.
\end{align*}
and further define $\hat{\Gamma}:D^{\infty,\rho}_t \to D^\infty_t$ by
\[\hat{\Gamma}(x)_i(t) = \begin{cases}
\Gamma(x)_i(t) & t < t^* \\ \Gamma(x)_i(t^*) & t
\ge t^*,
\end{cases}\]
for all $i \ge 0$.

By construction, $\hat{\Gamma}(x)_i(t) \in [-\eta-1,\eta+1]$ for all $i \ge 0$,
and thus $\hat{\Gamma}: D^{\infty,\rho}_t \to D^{\eta+1}_t$.
Further note that if $\eta = \infty$, then $g^\eta = 0$ and
\begin{align*}
\vn{\hat{\Gamma}(x)}_{\rho,t} &\le \vn{b}_\rho + \vn{y}_{\rho,t} + 
\sum_{i \ge 1}\rho^{-i}\lambda d t\left(\vn{x_i}_t + \vn{x_{i-1}}_t\right)  \\
&\quad\quad + \sum_{i \ge 1}\rho^{-i} t\left(\vn{x_{i+1}}_t + \vn{x_i}_t\right) \\
&\le \vn{b}_\rho + \vn{y}_{\rho,t} 
+ \lambda d t\sum_{i \ge 1}\rho^{-i}\vn{x_{i-1}}_t \\
&\quad\quad + t(\lambda d + 1)\sum_{i \ge 1}\rho^{-i} \vn{x_i}_t \\
&\quad\quad + t\sum_{i \ge 1}\rho^{-i}\vn{x_{i+1}}_t.
\end{align*}
We bound each of these sums individually. Observe
\begin{align*}
  \sum_{i \ge 1}\rho^{-i}\vn{x_{i-1}}_t &= 
    \rho^{-1}\sum_{i \ge 1}\rho^{-(i-1)}\vn{x_{i-1}}_t  \\
&= \rho^{-1}\sum_{i \ge 0}\rho^{-i}\vn{x_{i}}_t = \rho^{-1}\vn{x}_{\rho,t}.
\end{align*}
Similarly,
\[\sum_{i \ge 1}\rho^{-i} \vn{x_i}_t \le \sum_{i \ge 0}\rho^{-i}\vn{x_i}_t = \vn{x}_{\rho,t},\]
and
\begin{align*}
\sum_{i \ge 1}\rho^{-i}\vn{x_{i+1}}_t &= \rho \sum_{i \ge 1}\rho^{-(i+1)}\vn{x_{i+1}}_t \\
&= \rho \sum_{i \ge 2}\rho^{-i}\vn{x_i}_t \\
&\le \rho \sum_{i \ge 0}\rho^{-i}\vn{x_i}_t = \rho \vn{x}_t.
\end{align*}
Therefore we have
\[\vn{\hat{\Gamma}(x)}_{\rho,t} \le t\left(\lambda d\rho^{-1} + \lambda d + 1 + \rho\right) \vn{x}_{\rho,t} < \infty,\]
and thus $\hat{\Gamma} : D^{\infty,\rho}_t \to D^{\eta + 1,\rho}_t$.

We now show that for
\[t_0 < {1 \over \lambda d(1+\rho^{-1})(1+4^d) + 1 +
\rho},\]
$\hat{\Gamma}$ is a contraction on $D^{\eta+1,\rho}_{t}$ for all $t \le
t_0$.
Namely, we claim that there exists $\gamma < 1$ such that
for all $t \in [0,t_0]$ and $x^1,x^2 \in D^{\eta+1,\rho}_t$,
we have
\begin{equation}\label{eq:contraction-claim}
    \vn{\hat{\Gamma}(x^1)-\hat{\Gamma}(x^2)}_{\rho,t} \le
\gamma \vn{x^1-x^2}_{\rho,t}.
\end{equation}
Let $t \le t_0$ and $x^1,x^2 \in D^{\eta+1,\rho}_t$. 
We have for $i \ge 1$
\begin{align*}
\vn{\hat{\Gamma}(x^1)_i-\hat{\Gamma}(x^2)_i}_t
&\le \lambda d\int_0^t \bigg(
\vn{x_i^1-x_i^2}_s + \vn{x_{i-1}^1-x_{i-1}^2}_s \\
&\quad\quad\quad + \vn{g^\eta(x^1_i) - g^\eta(x^2_i)}_s 
+ \vn{g^\eta(x^1_{i-1}) - g^\eta(x^2_{i-1})}_s\bigg) ds \\
&\quad\quad + \int_0^t\left(\vn{x^1_{i+1}-x^2_{i+1}}_s +
\vn{x^1_i-x^2_i}_s\right)ds. \\
\end{align*}
By Lemma \ref{thm:g^m-lipschitz}, for $\eta < \infty$,
$g^\eta$ is Lipschitz when restricted to
$D^{\eta+2,\rho}_t$, with constant $4^d$. For $\eta = \infty$,
$g^\infty = 0$. Thus we now have
\begin{align*}
\vn{\hat{\Gamma}(x^1)_i-\hat{\Gamma}(x^2)_i}_t &\le t \lambda
d(1+4^d) \vn{x^1_{i-1}-x^2_{i-1}}_t \\
&\quad\quad + t\left(\lambda d(1+4^d) + 1\right) \vn{x^1_i-x^2_i}_{\rho,t}
+t \vn{x^1_{i+1}-x^2_{i+1}}_t.
\end{align*}
This implies
\begin{align*}
 \vn{\hat{\Gamma}(x^1)- \hat{\Gamma}(x^2)}_{\rho,t} &\le t\lambda
d(1+4^d) \sum_{i \ge 1}\rho^{-i} \vn{x^1_{i-1}- x^2_{i-1}}_{t} \\
&\quad\quad + t\left(\lambda d(1+4^d) + 1\right) \sum_{i \ge
1}\rho^{-i}\vn{x^1_i-x^2_i}_t \\
&\quad\quad + t\sum_{i \ge 1}\rho^{-i}\vn{x^1_{i+1}-x^2_{i+1}}_t.
\end{align*}
Reindexing and bounding these sums individually gives us
\begin{equation}\label{eq:contraction-bound}
     \vn{\hat{\Gamma}(x^1)- \hat{\Gamma}(x^2)}_{\rho,t} \le 
t\left(\lambda d(1 + \rho^{-1})(1+4^d) + 1 + \rho\right)
\vn{x^1-x^2}_{\rho,t}.
\end{equation}
Let 
\[t_0 < {1 \over \lambda d(1 + \rho^{-1})(1+4^d) + 1 + \rho}.\]
Then \eqref{eq:contraction-claim} holds with 
\[\gamma = t_0 \left(\lambda d(1 + \rho^{-1})(1+4^d) + 1 + \rho\right).\]
By the contraction mapping principle, $\hat{\Gamma}$ has a unique fixed point
$\hat{T}$ on $D^{\eta+1,\rho}_{t}$ such that $\hat{\Gamma}(\hat{T}) =
\hat{T}$. This fixed point provides a unique solution $\hat{T}$ to
\eqref{eq:integral-stopped} for $t \in [0,t_0]$. 

Suppose this fixed solution $\hat{T}$ is such that $t^* < t_0$. Then $\hat{T}$
is uniquely defined for all $t \ge 0$ and the proof is complete. Otherwise,
observe for $t \ge 0$ and $i \ge 1$ we have
\begin{align*}
    T_i(t) 
&= T_i(t_0) + y_i(t)-y_i(t_0) \\
&\quad\quad - \lambda d\int_{t_0}^{t}\left(T_i(s) - T_{i-1}(s) -
g^\eta(T_i(s)) + g^\eta(T_{i-1}(s)) \right)ds  \\
&\quad\quad + \int_{t_0}^{t}\left(T_{i+1}(s) - T_i(s)\right)ds.
\end{align*}
Thus if we define a shifted version of $y$ by $\tilde{y}(u) = y(u + t_0)$, then
$T_i(t)$ for $t = u + t_0$ and $u \ge 0$ is the solution to the system
\begin{align*}
x_0(u) &= 0 \\
x_i(u) &= T_i(t_0) - y_i(t_0) + \tilde{y}_i(u) \\
&\quad\quad - \lambda d\int_{0}^{u}\left(x_i(s) - x_{i-1}(s) -
g^\eta(x_i(s)) + g^\eta(x_{i-1}(s)) \right)ds  \\
&\quad\quad + \int_{0}^{u}\left(x_{i+1}(s) - x_i(s)\right)ds, \quad i \ge 1.
\end{align*}
Observe that this is the system \eqref{eq:integral-0}-\eqref{eq:integral-i}
with arguments
\[(T(t_0) - y(t_0), \tilde{y}, \lambda, \eta) \in \R^{\infty,\rho} \times
D^{\infty,\rho}_t \times \R \times \Rbar_{\ge 1},\]
and furthermore $x(0) = T(t_0) - y(t_0) + \tilde{y}(0) = T(t_0)$ and $T(t_0)
\in \bounded{\eta}$ because $t^* \ge t_0$. 
Thus we can repeat the contraction argument above to find
a unique solution $x$ for $u \in [0,t_0]$. This unique $x$ is the unique
solution $\hat{T}$ for $t \in [t_0,2t_0]$. If $t^* < 2t_0$, then $\hat{T}$ is 
uniquely defined for all $t \ge 0$ and the proof is complete.
Otherwise, the above extension argument can be repeated to find a unique
solution $\hat{T}$ for $[2t_0,3t_0],[3t_0,4t_0], \ldots$. If $t^* < k t_0$ for
some $k \ge 3$ then the argument stops there and we conclude $\hat{T}$ is
uniquely defined for all $t \ge 0$. Otherwise it may be extended to any $t \ge
0$.
\end{proof}

Before proving continuity, we state and prove a lemma bounding the growth 
of solutions to \eqref{eq:integral-stopped}.
\begin{lemma}\label{thm:uniform-bound}
For any $t \ge 0$, if $x$ is the solution to \eqref{eq:integral-stopped} for
arguments $(b,y,\lambda,\eta)$, then
\[\vn{x}_{\rho,t} \le \left(\vn{b}_\rho + \vn{y}_{\rho,t}\right) 
e^{\left(\lambda d(1+4^d)\left(1 + \rho^{-1}\right) + 1 + \rho\right)t}\]    
\end{lemma}

\begin{proof}
   We have
\begin{align*}
\vn{x_i}_{t} &\le |b_i| + \vn{y_i}_t + \lambda d\int_0^t\left(\vn{x_{i-1}}_s + \vn{x_i}_s + \vn{g(x_{i-1})}_s + \vn{g(x_i)}_s\right) ds \\
&\quad\quad + \int_0^t\left(\vn{x_i}_s + \vn{x_{i+1}}_s\right)ds,    
\end{align*} 
and thus
\begin{align*}
\vn{x}_{\rho,t} &\le \vn{b}_\rho + \vn{y}_{\rho,t} 
+ \lambda d(1+4^d)\sum_{i \ge 1}\rho^{-i}\int_0^t\left( \vn{x_{i-1}}_s + \vn{x_i}_s\right)ds \\
&\quad\quad+ \sum_{i \ge 1}\rho^{-i} \int_0^t\left(\vn{x_i}_s + \vn{x_{i+1}}_s\right)ds \\
&\le  \vn{b}_\rho + \vn{y}_{\rho,t} + 
\left(\lambda d(1+4^d)\left(1 + \rho^{-1}\right) + 1 + \rho\right)
\sum_{i \ge 0}\rho^{-i}\int_0^t\vn{x_i}_sds \\
&= \vn{b}_\rho + \vn{y}_{\rho,t} + 
\left(\lambda d(1+4^d)\left(1 + \rho^{-1}\right) + 1 + \rho\right)
\int_0^t\sum_{i \ge 0}\rho^{-i}\vn{x_i}_sds \\
&= \vn{b}_\rho + \vn{y}_{\rho,t} + 
\left(\lambda d(1+4^d)\left(1 + \rho^{-1}\right) + 1 + \rho\right)
\int_0^t\vn{x}_{\rho,s}ds.
\end{align*}
By Gronwall's inequality (Lemma \ref{thm:gronwall}), we have
\[\vn{x}_{\rho,t} \le \left(\vn{b}_\rho + \vn{y}_{\rho,t}\right) 
e^{\left(\lambda d(1+4^d)\left(1 + \rho^{-1}\right) + 1 + \rho\right)t}.\]
\end{proof}

\begin{proof}[Proof of Theorem \ref{thm:const-integral}: continuity]
We now prove continuity for $f$ restricted to the domain $\region{K}$ for any $K > 0$. 
Suppose $(b^n,y^n,\lambda^n,\eta^n) \to
(b,y,\lambda,\eta)$ with respect to the product topology, 
with $(b^n,y^n,\lambda^n,\eta^n) \in \region{K}$ and since
the set $\region{K}$ is closed, we also have
$(b,y,\lambda,\eta) \in \region{K}$.
Suppose $x^n$ is the unique solution to \eqref{eq:integral-stopped} for
$(b^n,y^n,\lambda^n,\eta^n)$ and  $x$ is the unique solution for
$(b,y,\lambda,\eta)$. 
Let $i^*_n = {\alpha \over 2} \log_\rho \eta_n$ and $i^* = {\alpha \over 2}
\log_\rho \eta$. Let $t^*_n$ and $t^*$ be the stopping times for $x^n$ and $x$,
respectively.

Recall that by the definition of $\region{K}$, we have $x^n(0) = b^n + y^n(0)
\in \bounded{\eta^n}$.
Note that this, along with the definition of $t^*_n$, implies that 
$x^n(t) \in \bounded{\eta^n}$ for any $t \ge 0$.
Similarly, we have $x(t) \in \bounded{\eta}$ for any $t \ge 0$.
We adopt the simplified notation $g^n \defeq g^{\eta^n}$ and $g = g^\eta$. 

We will show $x^n \to x$ in $D^{\infty,\rho}_t$ equipped with the product topology. Fix $\epsilon > 0$ and $i_0 \in \N$.

Recall the map $\hat{\Gamma}$ defined in the proof of existence and uniqueness above. We define $\hat{\Gamma}^n$ and $\hat{\Gamma}$ analogously for $(b^n,y^n,\lambda^n,\eta^n)$ and $(b,y,\lambda,\eta)$, respectively. Define
\[t_0 = {1 \over 2\left(1+\lambda\right)d(1+\rho^{-1})(1+4^d) + 1 + \rho},\]
and
\[\gamma = t_0 \left(\left(1+\lambda\right)d(1+\rho^{-1})(1+4^d) + 1 + \rho\right) = 1/2.\]
By \eqref{eq:contraction-bound}, for $0 \le t \le t_0$ and $x^1,x^2 \in D^{\infty,\rho}_t$, we have
\begin{align*}
\vn{\hat{\Gamma}(x^1)-\hat{\Gamma}(x^2)}_{\rho,t} &\le 
t_0\left(\left(1+\lambda\right)d(1+\rho^{-1})(1+4^d) + 1 + \rho\right) 
    \vn{x^1-x^2}_{\rho,t} \\
&= \gamma \vn{x^1-x^2}_{\rho,t}.
\end{align*}
Furthermore, because $\lambda^n \to \lambda$, there exists some $N_\lambda$ such that $\lambda^n < \lambda + 1$ for $n \ge N_\lambda$, and thus for such $n$,
\begin{align*}
    \vn{\hat{\Gamma}^n(x^1)-\hat{\Gamma}^n(x^2)}_{\rho,t} &\le t_0 \left(\lambda^n d(1+\rho^{-1})(1+4^d) + 1 + \rho\right) \vn{x^1-x^2}_{\rho,t} \\
&\le  t_0 \left(\left(1+\lambda\right)d(1+\rho^{-1})(1+4^d) + 1 + \rho\right) \vn{x^1-x^2}_{\rho,t} \\
&= \gamma \vn{x^1-x^2}_{\rho,t}.
\end{align*}
Thus for $0 \le t \le t_0$, $\hat{\Gamma}$ and $\hat{\Gamma}^n$ for 
$n \ge N_\lambda$ are contractions on the space $D^{\infty,\rho}_t$ with
coefficient $\gamma = 1/2$. Recall that $x$ and $x^n$ are the fixed points
of $\hat{\Gamma}$ and $\hat{\Gamma}^n$, respectively, and that therefore
each can be found by repeated iteration of an arbitrary point in 
$D^{\infty,\rho}_t$. Specifically, we define
\begin{align*}
    x^0 &= 0 \quad &x^{n,0} &= 0 \\
    x^r &= \hat{\Gamma}(x^{r-1}) \quad &x^{n,r} &= 
        \hat{\Gamma}^n(x^{n,r-1}), \quad r \ge 1.
\end{align*}
Then the following inequalities hold:
\begin{align}
\vn{x - x^r}_{\rho,t} &\le {\gamma^n \over 1 - \gamma} \vn{x^1-x^0}_{\rho,t}
 = 2^{-r+1}\vn{x^1}_{\rho,t}, \notag\\ 
\vn{x^n - x^{n,r}}_{\rho,t} &\le {\gamma^n \over 1 - \gamma} 
\vn{x^{n,1}-x^{n,0}}_{\rho,t} = 2^{-r+1}\vn{x^{n,1}}_{\rho,t}. 
\label{eq:contraction-n}
\end{align}
Observe
\begin{align*}
    \vn{x^1}_{\rho,t} &= \vn{\hat{\Gamma}(0)}_{\rho,t} = 
\vn{b + y}_{\rho,t} \le \vn{b}_{\rho} + \vn{y}_{\rho,t} \le 2K,
\end{align*}
and similarly
\[\vn{x^{n,1}}_{\rho,t} \le \vn{b^n}_{\rho} + \vn{y^n}_{\rho,t} \le 2K.\]

We now argue that there exists $N$ such that for all $n \ge N$ and 
$i \le i_0$, $\vn{x_i^n - x_i}_t < \epsilon$ for $t \le t_0$. This will
establish continuity of $f$ with respect to the product topology 
for such $t$. Observe
\begin{equation}\label{eq:continuity-breakdown}
\vn{x_i^n - x_i}_{t} \le \vn{x_i^n - x_i^{n,r}}_t + \vn{x_i^{n,r} - x_i^r}_t 
+ \vn{x_i^r - x_i}_t.
\end{equation}
We bound these three terms individually. By \eqref{eq:contraction-n}, 
we have
\begin{align*}
    \vn{x_i^n - x_i^{n,r}}_t &\le \rho^i\vn{x^n-x^{n,r}}_{\rho,t} \\
&\le 2^{-r+1}\rho^i\left(\vn{b^n}_\rho + \vn{y^n}_{\rho,t}\right) \\
&\le 2^{-r+2}K\rho^{i_0}.
\end{align*}
For
\begin{equation}\label{eq:r-bound}
r > 2 + \log_2{3\rho^{i_0}K \over \epsilon},
\end{equation}
we have
\begin{equation}\label{eq:bounded-nr}
  \vn{x_i^n - x_i^{n,r}}_t < \epsilon / 3 \quad \text{ for all } i \le i_0.  
\end{equation}
Via a similar argument, for $r$ satisfying \eqref{eq:r-bound}, we also have
\begin{equation}\label{eq:bounded-limitr}
  \vn{x_i - x_i^r}_t < \epsilon / 3 \quad \text{ for all } i \le i_0.  
\end{equation}
Finally, we consider
\begin{align*}
\vn{x_i^{n,r+1} - x_i^{r+1}}_t &\le |b_i^n - b_i| + \vn{y_i^n - y_i}_t \\
&\quad\quad + \lambda^n d \int_0^t \big(\vn{x^{n,r}_{i-1} - x^r_{i-1}}_s 
+ \vn{x^{n,r}_i - x^r_i}_s \\
&\quad\quad\quad + \vn{g^n(x^{n,r}_{i-1}) - g(x^r_{i-1})}_s 
+ \vn{ g^n(x^{n,r}_i) - g(x^r_i)}_s \big)ds \\
&\quad\quad + |\lambda - \lambda^n|
d\int_0^t\big(\vn{x^r_{i-1}}_s + \vn{x^r_i}_s \\
&\quad\quad\quad + \vn{g(x^r_{i-1})}_s + \vn{g(x^r_i)}_s\big)ds \\
&\quad\quad + \int_0^t\left(\vn{x^{n,r}_{i} - x^r_i}_s +
\vn{x^{n,r}_{i+1}-x^r_{i+1}}_s\right)ds.   
\end{align*}
By Lemma \ref{thm:g^m-lipschitz}, $g$ is Lipschitz continuous when 
restricted to $D^{\eta+2,\rho}_t$, so we have
\begin{align*}
\vn{g(x^r_{i-1})}_s + \vn{g(x^r_i)}_s &\le 4^d\vn{x^r_{i-1}}_s + 
4^d\vn{x^r_i}_s.
\end{align*}
Further note the bound
\[\vn{g^n(x_i^{n,r}) - g(x_i^r)}_s \le \vn{g^n(x_i^{n,r}) - g(x_i^{n,r})}_s
+ \vn{g(x_i^{n,r}) - g(x_i^r)}_s.\]
Using these pieces, crudely bounding integrals for some terms, 
and rearranging we have
\begin{align*}
\vn{x_i^{n,r+1} - x_i^{r+1}}_t &\le |b_i^n - b_i| + \vn{y_i^n - y_i}_t \\
&\quad\quad + t\vn{g^n(x^{n,r}_{i-1}) - g(x^{n,r}_{i-1})}_t 
+ t\vn{g^n(x^{n,r}_{i}) - g(x^{n,r}_{i})}_t \\
&\quad\quad + |\lambda - \lambda^n|t(1+4^d)\left(\vn{x^r_{i-1}}_t + \vn{x^r_i}_s\right) \\
&\quad\quad + \lambda^n d \int_0^t 
\left(\vn{x^{n,r}_{i-1} - x^r_{i-1}}_s 
+ \vn{x^{n,r}_i - x^r_i}_s \right)ds \\
&\quad\quad + \lambda^n d \int_0^t 
\left(\vn{g(x^{n,r}_{i-1}) - g(x^r_{i-1})}_s 
+ \vn{g(x^{n,r}_i) - g(x^r_i)}_s \right)ds \\
&\quad\quad + \int_0^t\left(\vn{x^{n,r}_{i} - x^r_i}_s +
\vn{x^{n,r}_{i+1}-x^r_{i+1}}_s\right)ds.   
\end{align*}
Recall that for any $t \ge 0$ we have $x^{n,r}(t) \in \bounded{\eta^n}$
and thus
$\left|x^{n,r}_i(t)\right| \le \left(\eta^n\right)^{\alpha}$
for $i \le i^*_n$, and $\left|x^{n,r}_i(t)\right| \le \eta^n +1$ for $i >
i^*_n$, so the conditions of Lemma \ref{thm:g^m-converge} are satisfied and
therefore for $i \le i_0$ we have 
$\vn{g^n(x_{i}^{n,r}) - g(x_i^{n,r})}_t \to 0$, and similarly 
$\vn{g^n(x_{i-1}^{n,r}) - g(x_{i-1}^{n,r})}_t \to 0$. This, along with 
$(b^n,y^n,\lambda^n) \to (b,y,\lambda)$ implies that for any $\delta > 0$ we 
can choose $N_\delta$ such that for $n \ge N_\delta$ we have
\[\lambda_n \le \lambda + 1,\]
\[\eta^n \le \eta + 1,\]
and for all $i \le i_0 + r + 1$,
\begin{align*}
 &|b_i^n - b_i| + \vn{y_i^n - y_i}_t \\
&\quad + t\vn{g^n(x^{n,r}_{i-1}) - g(x^{n,r}_{i-1})}_t 
+ t\vn{g^n(x^{n,r}_{i}) - g(x^{n,r}_{i})}_t \\
&\quad + |\lambda - \lambda^n|t(1+4^d)\left(\vn{x^r_{i-1}}_t 
    + \vn{x^r_i}_t\right) \quad\quad < \delta.
\end{align*}
Observe that $\eta^n \le \eta + 1$ implies $x_i^{n,r} \in D^{\eta + 2,\rho}_t$,
so
\[\vn{g(x_i^{n,r}) - g(x_i^r)}_s \le 4^d\vn{x_i^{n,r} - x_i^r}_s.\]
For $n \ge N_\delta$, we have for all $i \le i_0 + r + 1$
\begin{align*}
\vn{x_i^{n,r+1} - x_i^{r+1}}_t &< \delta + 
\int_0^t\left(\vn{x^{n,r}_{i} - x^r_i}_s + 
\vn{x^{n,r}_{i+1}-x^r_{i+1}}_s\right)ds \\
&\quad\quad + (1+4^d)(1+\lambda) d \int_0^t 
\left(\vn{x^{n,r}_{i-1} - x^r_{i-1}}_s 
+ \vn{x^{n,r}_i - x^r_i}_s \right)ds.
\end{align*}
We rewrite this as
\begin{align*}
\vn{x_i^{n,r+1} - x_i^{r+1}}_t &< \delta + 
C
\int_0^t \max_{i-1 \le j \le i+1}\vn{x^{n,r}_j - x^r_j}_s ds,
\end{align*}
where $C = \left(2(1+4^d)(1+\lambda) d + 2\right)$. For $i \le i_0$, 
this can be expanded as
\begin{align*}
  \vn{x_i^{n,r+1} - x_i^{r+1}}_t &< \delta + Ct \max_{i-1 \le j \le i+1}\vn{x^{n,r}_j - x^r_j}_t \\
&< \delta + Ct\left(\delta + 
Ct\max_{i-2 \le j \le i+2}\vn{x^{n,r-1}_j - x^{r-1}_j}_t\right) \\
&< \delta \sum_{k = 0}^{r}(Ct)^k + (Ct)^{r+1}\max_{(i - r-1) \wedge 0 \le j \le i+r+1}\vn{x_j^{n,0} - x^0_j}_t
\end{align*}
Recall that $x^{n,0} = x^0 = 0$, so for all $i \ge 0$ we have 
$\vn{x_i^{n,0} - x_i^0}_t = 0$, and thus
\begin{align*}
    \vn{x_i^{n,r+1} - x_i^{r+1}}_t < \delta {(Ct)^{r+1}-1 \over Ct - 1}.
\end{align*}
Reindexing gives, for all $i \le i_0$,
\[ \vn{x_i^{n,r} - x_i^{r}}_t < \delta {(Ct)^{r}-1 \over Ct - 1},\]
and thus for 
\[\delta < {\epsilon \over 3} \cdot {Ct - 1 \over (Ct)^{r}-1 },\]
we have
\begin{equation}\label{eq:bounded-diff}
    \vn{x_i^{n,r} - x_i^{r}}_t < \epsilon / 3.
\end{equation}
Thus by plugging \eqref{eq:bounded-nr},\eqref{eq:bounded-limitr}, and 
\eqref{eq:bounded-diff} into \eqref{eq:continuity-breakdown}, if we choose $r$ and $n$ such that 
\[r > 2 + \log_2{3\rho^{i_0}K \over \epsilon}, \quad \text{ and } \quad 
n \ge \max\left(N_\lambda,N_\delta\right),\]
we have for all $i \le i_0$,
\[\vn{x_i^n - x_i}_t < \epsilon,\]
establishing the continuity of $f$ for $t \le t_0$.

As in the proof of existence and uniqueness, we define a shifted version of
$y$ by $\tilde{y}(u) = y(u+t_0)$ observe that $x(t)$ for $t = u + t_0$ and
$u \ge 0$ is the solution $\hat{z}$ to the system
\begin{align*}
z_0(u) &= 0 \\
z_i(u) &= x_i(t_0) - y_i(t_0) + \tilde{y}_i(u) \\
&\quad\quad - \lambda d\int_{0}^{u}\left(x_i(s) - x_{i-1}(s) -
g^\eta(x_i(s)) + g^\eta(x_{i-1}(s)) \right)ds  \\
&\quad\quad + \int_{0}^{u}\left(x_{i+1}(s) - x_i(s)\right)ds, \quad i \ge 1 \\
\hat{z}(u) &= \begin{cases} z(u) & u < u^* \\ z(u^*) & u \ge u^*, \end{cases} 
\end{align*}
where $u^*$ is defined analogously to $t^*$ in 
\eqref{eq:stopping-time-full}.
Observe that this is the system 
\eqref{eq:integral-0}-\eqref{eq:integral-i},
\eqref{eq:integral-stopped} with arguments
\[(x(t_0) - y(t_0), \tilde{y}, \lambda, \eta) \in \R^{\infty,\rho} \times
D^{\infty,\rho}_t \times \R \times \Rbar_{\ge 1},\]
and furthermore $\hat{z}(0) = x(t_0) - y(t_0) + \tilde{y}(0) = x(t_0) 
\in \bounded{\eta}$. Also 
$\vn{\tilde{y}}_{\rho,u} \le \vn{y}_{\rho,u+t_0} < K$ 
and by Lemma \ref{thm:uniform-bound} we have
\begin{align*}
 \vn{x(t_0)-y(t_0)}_{\rho,t} &\le \vn{y(t_0)}_\rho + \vn{x(t_0)}_\rho \\
&\le K + \left(\vn{b}_\rho + \vn{y}_{t_0,\rho}\right)
e^{\left(\lambda d(1+4^d)\left(1 + \rho^{-1}\right) + 1 + \rho\right)t_0} \\
&\le K + 
2Ke^{\left((1+\lambda)d(1+4^d)\left(1 + \rho^{-1}\right) 
+ 1 + \rho\right)t_0}.
\end{align*} 
We define
\[K_1 \defeq K + 
2Ke^{\left((1+\lambda)d(1+4^d)\left(1 + \rho^{-1}\right) 
+ 1 + \rho\right)t_0},\]
so
\[(x(t_0) - y(t_0), \tilde{y}, \lambda, \eta) \in \region{K_1}.\]
 A similar construction allows us to define 
$\tilde{y}^n$ and $\hat{z}^n$, and we have $\tilde{y}^n \to \tilde{y}$ in 
$D^{\infty,\rho}_u$ for any $u \ge 0$. 
For sufficiently large $n$ we have $\lambda^n < \lambda + 1$, and thus 
\begin{align*}
 \vn{x^n(t_0)-y^n(t_0)}_{\rho,t} &\le K + 
2Ke^{\left(\lambda^nd(1+4^d)\left(1 + \rho^{-1}\right) 
+ 1 + \rho\right)t_0} \\
&\le K + 
2Ke^{\left((1+\lambda)d(1+4^d)\left(1 + \rho^{-1}\right) 
+ 1 + \rho\right)t_0} = K_1,
\end{align*}
so 
\[(x^n(t_0) - y^n(t_0),\tilde{y}^n,\lambda^n,\eta^n) \in \region{K_1}.\]
Therefore we can repeat the continuity argument above to show 
$\hat{z}^n \to \hat{z}$ for $u \le t_0$, which implies 
$x^n \to x$ for $t \in [t_0,2t_0]$. 
This extension argument
can be repeated to prove $x^n \to x$ for $[2t_0,3t_0],[3t_0,4t_0], \ldots$.
Thus for any $t \ge 0$ we have $x^n \to x$ in $D^{\infty,\rho}_t$ and thus
$f$ is continuous.
\end{proof}

\section{Martingale representation.}\label{sec:martingales}
We now show that the stochastic process underlying the supermarket system
stopped at some appropriate time can be
written in a form that exactly matches that of $\hat{T}$ in 
\eqref{eq:integral-stopped}. 
This will allow us to use Theorem \ref{thm:const-integral} 
to prove Theorem \ref{thm:const-result} in the next section.

Before introducing the stopped variant, we will consider the original 
supermarket system and show that it can be represented by the equations
\eqref{eq:integral-0}-\eqref{eq:integral-i} for a particular choice of
arguments $(b,y,\lambda, \eta)$. 

For $i \ge 1$, recall the representation \eqref{eq:s-form-i}.
Given the definition $T_i^n = \eta_n\left(1 - S_i^n\right)$ we
can rewrite this as
\begin{align}
T_i^n(t) &= T_i^n(0) - {\eta_n \over n}A_i\left(\lambda_n n
\int_0^t\left(\left(1 - {1 \over \eta_n}T_{i-1}^n(s)\right)^d-\left(1 - 
{1 \over \eta_n}T_i^n(s)\right)^d\right)ds\right) \notag\\
&\quad\quad + {\eta_n \over n} D_i\left( n \int_0^t\left(\left(1-{1 \over
\eta_n}T_i^n(s)\right) - \left(1 -
{1\over\eta_n}T_{i+1}^n(s)\right)\right)ds \right) \notag\\
&= T_i^n(0) - {\eta_n \over n}A_i\bigg(\lambda_n n
\int_0^t\bigg(\left(1 - {d \over
\eta_n}T_{i-1}^n(s) + {d \over \eta_n} g^{\eta_n}(T_{i-1}^n(s))\right) \notag\\
&\hspace{0.25\textwidth} -\left(1 - {d
\over \eta_n}T_i^n(s) + {d \over \eta_n}
g^{\eta_n}(T_i^n(s))\right)\bigg)ds\bigg)
\notag\\
&\quad\quad + {\eta_n \over n} D_i\left( {n \over \eta_n}
\int_0^t\left(T_{i+1}^n(s) - T_{i}^n(s)\right)ds \right)
\notag\\
&= T_i^n(0) - {\eta_n \over n}A_i\left(\lambda_n d
{n \over \eta_n}\int_0^t\left(T_{i}^n(s) - T_{i-1}^n(s) - g^{\eta_n}(T_i^n(s))
+ g^{\eta_n}(T_{i-1}^n(s))\right)ds\right) + \notag\\
&\quad\quad + {\eta_n \over n}D_i\left({n \over \eta_n}
\int_0^t\left(T_{i+1}^n(s) - T_i^n(s)\right)ds\right),
\label{eq:T-form-i-no-martingales}
\end{align}
where $g^{\eta_n}$ is defined as in \eqref{eq:g^n}. We now define scaled
martingale processes
\begin{align}
M_i^n(t) &= {\eta_n \over n}A_i\left(\lambda_n d
{n \over \eta_n} \int_0^t\left(T_{i}^n(s) - T_{i-1}^n(s) + g^{\eta_n}(T_i^n(s))
- g^{\eta_n}(T_{i-1}^n(s))\right)ds\right) \notag\\
&\quad\quad - \lambda_n d
\int_0^t\left(T_{i}^n(s) - T_{i-1}^n(s) + g^{\eta_n}(T_i^n(s)) -
g^{\eta_n}(T_{i-1}^n(s))\right)ds, \label{eq:martingale-arrival}\\
N_i^n(t) &= {\eta_n \over n}D_i\left({n \over \eta_n}\int_0^t\left(T_{i+1}^n(s)
- T_{i}^n(s)\right)ds\right) - \int_0^t\left(T_{i+1}^n(s) -
T_{i}^n(s)\right)ds. \label{eq:martingale-departure}
\end{align}
Now we can rewrite the system for $i \ge 1$ as
\begin{align}
T_i^n(t) &=	T_i^n(0) - M_i^n(t) - \lambda_n d\int_0^t\left(T_i^n(s) -
T_{i-1}^n(s) + g^{\eta_n}(T_i^n(s)) - g^{\eta_n}(T_{i-1}^n(s)\right)ds \notag\\
&\quad\quad + N_i^n(t) + \int_0^t\left(T_{i+1}^n(s) - T_i^n(s)\right)ds.
\label{eq:T-form-i}
\end{align}
This representation matches \eqref{eq:integral-i} with $b = T^n(0)$, $y = -M^n
+ N^n$, $\lambda = \lambda_n$ and $\eta = \eta_n$. 

Recall that the assumptions of Theorem \ref{thm:const-result} include constants
$\rho > 1$ and $0 < \alpha < 1/2$.
We now define $i^*_n = {\alpha \over 2} \log_\rho \eta_n$ and define a stopping
time
\[t^*_n = \inf\left\{t \ge 0 : \exists\text{ } i \st 1 \le i \le i^*_n
\text{ and } T_i^n(t) \ge \eta_n^{\alpha}\right\}.\]
Note that compared to \eqref{eq:stopping-time-full}, this stopping time does
not contain terms checking $T_i^n(t) \le -(\eta_n)^{\alpha}$, or
$\left|T_i^n(t)\right| \ge \eta_n + 1$. This is because
$0 \le T_i^n(t) \le \eta_n$ for all $i \ge 0$ and for all $t \ge 0$, so such
conditions are never met. Thus $t^*_n$ is equivalent to the stopping time
defined by \eqref{eq:stopping-time-full}.
We consider the process $\hat{T}^n$ defined by
\[\hat{T}_i^n(t) = \begin{cases} T_i^n(t) & t < t^*_n \\ T_i^n(t^*_n) & t
\ge t^*_n.\end{cases}\]

As noted above, the stopped supermarket model $\hat{T}^n$ is the unique
solution of the
integral equation system \eqref{eq:integral-stopped} described in Theorem
\ref{thm:const-integral}, with arguments $b = T^n(0)$, $y = -M^n + N^n$,
$\lambda = \lambda_n$, and $\eta = \eta_n$.

\section{Martingale convergence.}\label{sec:martingale-convergence}
Because our sequence of supermarket models indexed by $n$ are all examples of
the integral equation system \eqref{eq:integral-stopped}, and Theorem
\ref{thm:const-integral} shows that this system defines a continuous map from
arguments $(b,y,\lambda,\eta) \in \region{K}$ for some $K > 0$
to the stopped system
$f(b,y,\lambda,\eta) = \hat{T}$, we will find the weak limit of the finite
system $\hat{T}^n$ by finding the limits of the arguments $(T^n(0), -M^n + N^n, \lambda_n, \eta_n)$.
Though we do not use the continuous mapping theorem because we do not have
$(T^n(0), -M^n + N^n, \lambda_n, \eta_n) \in \region{K}$ almost surely for any non-random $K$,
the proof will still rely on the continuity of $f$ and the limits of the arguments.

Three of these limits are given, as
\eqref{eq:const-starting-state} provides $\hat{T}^n(0) \Rightarrow T(0)$,
and we have $\lambda_n \to 1$ and $\eta_n \to \infty$. 
We claim $-M^n + N^n \Rightarrow 0$. 
We now prove the following:
\begin{prop}\label{thm:martingale-limit}
    For $M^n$ and $N^n$ as defined in
\eqref{eq:martingale-arrival}-\eqref{eq:martingale-departure}, if
the assumptions of Theorem \ref{thm:const-result} hold, then
\[\E \vn{M^n}_{\rho,t}, \E \vn{N^n}_{\rho,t} \to 0 \quad \text{ as } n \to \infty.\]
This implies $M^n, N^n \Rightarrow 0$ in $D^{\infty,\rho}_t$ equipped with the product topology
\end{prop}
Before proving this proposition, we prove a bound on the unbounded
$\rho$-norm of $T^n$ in expectation.
\begin{lemma}\label{thm:const-expectation-finite}
For any $\gamma > 1$ we have
\begin{equation}\label{eq:const-expectation-finite}
\E\left[\vn{T^n}_{\gamma,t}\right] \le \E\left[\vn{T^n(0)}_\gamma\right]
e^{\gamma t}.
\end{equation}
\end{lemma}

\begin{proof}
By dropping negative terms from \eqref{eq:T-form-i-no-martingales} we
have
\[T_i^n(t) \le T_i^n(0) + {\eta_n
\over n}D_i\left({n \over \eta_n}\int_0^tT_{i+1}^n(s)ds\right).\]
Then we have
\begin{align*}
\E\left[\vn{T_i^n}_t\right]  &\le \E T_i^n(0) +
\E\left[\int_0^tT_{i+1}^n(s)ds\right]
\\
&\le \E T_i^n(0) + \int_0^t\E\left[\vn{T_{i+1}^n}_s\right]ds.
\end{align*}
This implies
\begin{align*}
\E\left[\vn{T^n}_{\gamma,t}\right] &= \sum_{i \ge
1}\gamma^{-i}\E\left[\vn{T_i^n}_t\right] \\
&\le \E\left[\vn{T^n(0)}_\gamma\right] + \sum_{i \ge 1}
\gamma^{-i}\int_0^t\E\left[\vn{T_{i+1}^n}_s\right]ds \\
&\le \E\left[\vn{T^n(0)}_\gamma\right] +
\gamma\int_0^t\E\left[\vn{T^n}_{\gamma,s}\right]ds.
\end{align*}
We can now apply Lemma \ref{thm:gronwall} to conclude
\[\E\left[\vn{T^n}_{\gamma,t}\right] \le \E\left[\vn{T^n(0)}_\gamma\right]
e^{\gamma t},\]
as desired.
\end{proof}

\begin{proof}[Proof of Proposition \ref{thm:martingale-limit}]
We first prove the statement for $N^n$.
We begin with some observations about $N^n$ and introduce some additional
definitions. Let
\begin{equation}\label{eq:tau-N}
\tau^n_i \defeq \int_0^t\left(T_{i+1}^n(s) -
T_{i}^n(s)\right)ds, \quad i \ge 0.
\end{equation}
Observe for all $i \ge 0$
\begin{equation}\label{eq:tau-N-bound}
\tau^n_i \le t\vn{T_{i+1}^n}_t.
\end{equation}
Now observe
\begin{align}
    \vn{N^n_i}_t &= \sup_{0 \le u \le t}\left| {\eta_n \over n}D_i\left({n
\over \eta_n}\int_0^u\left(T_{i+1}^n(s)
- T_{i}^n(s)\right)ds\right) - \int_0^u\left(T_{i+1}^n(s) -
T_{i}^n(s)\right)ds\right| \notag\\
&= \sup_{0 \le s \le \tau^n_i}\left|{\eta_n \over n}D_i\left({n
\over \eta_n}s\right) - s\right| \notag\\
&= \sup_{0 \le u \le {n \tau^n_i \over \eta_n}} {\eta_n \over
n}\left|D_i(u)-u\right|. \label{eq:CLT-setup}
\end{align}

We claim
\begin{equation}\label{eq:kappa-exp-N-short}
\lim_{n \to \infty} \max_{1 \le i \le i^*_n}
\E\vn{N_i^n}_{t} = 0.  
\end{equation}
By Lemma \ref{thm:const-expectation-finite} with $\gamma = \rho$
we have
\[\E\left[\vn{T^n}_{\rho,t}\right] \le
\E\left[\vn{T^n(0)}_{\rho}\right] e^{\rho t}.\]
Recall \eqref{eq:starting-expectation-bounded}, fix some $\epsilon > 0$ and
let
\[C = \left(\limsup_{n \to
\infty}\E\left[\vn{T^n(0)}_{\rho}\right]+\epsilon\right) e^{\rho t}.\]
For all sufficiently large $n$ we have
\[\E\left[\vn{T^n}_{\rho,t}\right] \le C.\]
This further implies that for all $1 \le i \le i^*_n+1$ we have
\[\E\left[\vn{T_i^n}_t\right] \le C\rho^{ i} \le C\rho^{i^*_n+1}.\]
Define
\[\delta_n = {1 \over \left(\log_\rho \eta_n\right)^2},\]
and define the event
\[A = \left\{\exists \text{ } i \le i^*_n+1 \st \vn{T_i^n}_t \ge {1
\over \delta_n}C\rho^{i^*_n+1}\right\}.\]
By Markov's inequality and the union bound
\begin{equation}\label{eq:T-small-small-i}
    \PP\left(A\right) \le \delta_n (i^*_n+1) = {\alpha \over 2\log_\rho
\eta_n}
+ {1 \over \left(\log_\rho \eta_n\right)^2}.
\end{equation}
Observe
\begin{align}
     \E\left[\vn{N_i^n}_{t}\right] &=
\E\left[\vn{N_i^n}_t
\1\{A\}\right] + \E\left[\vn{N_i^n}_t \1\{A^c\}\right].
\label{eq:breakdown}
\end{align}
We bound these two terms separately. 
We first consider the second term of \eqref{eq:breakdown}. Recall
\eqref{eq:CLT-setup} and observe
\begin{equation}\label{eq:CLT-setup-A^c}
\vn{N_i^n}_t\1\{A^c\} \le \sup_{0 \le u \le {n \tau^n_i \over
\eta_n}}
{\eta_n \over n}\left|D_i(u)-u\right|.
\end{equation}
For $0 \le i \le i^*_n$, $A^c$ and \eqref{eq:tau-N-bound} imply
\begin{align*}
\tau^n_i &\le {t \over \delta_n}C \rho^{i^*_n+1 } \\
&= t\rho C {\eta_n^{\alpha/ 2} \over \delta_n}.
\end{align*}
Applying this for \eqref{eq:CLT-setup-A^c} we obtain
\[\vn{N_i^n}\1\{A^c\} \le \sup{\eta_n \over
n}\left|D_i(u)-u\right|,\]
where the supremum is over
\[0 \le u \le t\rho C n \delta_n^{-1} \eta_n^{{\alpha \over 2} - 1}.\]
We define
\begin{align*}
\nu_n &\defeq n \delta_n^{-1} \eta_n^{{\alpha \over 2} - 1} \\
    &= {n \over \eta_n^{1-\alpha / 2}}\left(\log_\rho \eta_n \right)^2.
\end{align*}
Recall that $\eta_n = \bo{\sqrt{n}}$, so 
\[\eta_n^{1 - \alpha / 2} = \bo{n^{1/2-\alpha /4}},\]
 so $\nu_n \to \infty$ as $n \to \infty$. 
Thus we have
\begin{align*}
    \vn{N_i^n}\1\{A^c\} &\le \sup_{0 \le u \le t\rho C
\nu_n}{\eta_n
\over n}\left|D_i(u) - u\right| \\
&= {\eta_n \sqrt{\nu_n} \over n}\sup_{0 \le u \le t\rho C
\nu_n}{1 \over \sqrt{\nu_n}}\left|D_i(u) - u\right|.
\end{align*}
Let
\[\gamma_n \defeq {\eta_n \sqrt{\nu_n} \over n}\E\sup_{0 \le u \le
t\rho C \nu_n}{1 \over \sqrt{\nu_n}}\left|D_i(u) - u\right|.\]
Thus $\E\left[\vn{N_i^n}_t\1\{A^c\}\right] \le
\gamma_n$ for $1 \le i \le i^*_n$.
By the Functional Central Limit Theorem (FCLT), since $\nu_n \to \infty$ we
have
\[\sup_{0 \le u \le t\rho C
\nu_n}{1 \over \sqrt{\nu_n}}\left|D_i(u) - u\right| \Rightarrow \sup_{0 \le u
\le t\rho C} |B(u)|,\]
where $B$ is a standard Brownian motion.
Furthermore, observe that
\begin{align*}
    {\eta_n \sqrt{\nu_n} \over n} &= n^{-1/2}\eta_n^{1/2+
\alpha /4}\log_\rho\eta_n.
\end{align*}
By assumption we have $\eta_n = \bo{\sqrt{n}}$, so 
\begin{align*}
\eta_n^{1/2+\alpha/4} &= \bo{n^{1/4 + \alpha / 8}}. 
\end{align*}
Recall that $\alpha < 1/2$ so $1/4 + \alpha / 8 < 1/2$
which implies 
\[{\eta_n \sqrt{\nu_n}\over n} \to 0,\]
and thus $\gamma_n \to 0$.

We now consider the first term of \eqref{eq:breakdown}. 
By the Cauchy-Schwarz inequality and \eqref{eq:T-small-small-i}, we have
\begin{align*}
    \E\left[\vn{N_i^n}_t \1\{A\}\right] &\le
\sqrt{\E\left[\vn{N_i^n}_t^{2}\right]}\sqrt{\PP
(A)} 
\end{align*}
By \eqref{eq:tau-N-bound} and $T_i(t) \le \eta_n$ for all $i \ge 1$, we have
\begin{align*}
    \vn{N_i^n}_t &\le \sup_{0 \le u \le nt} 
{\eta_n \over n}\left|D_i(u)-u\right| \\
&= {\eta_n \over \sqrt{n}}\sup_{0 \le u \le nt} 
{1 \over \sqrt{n}}\left|D_i(u)-u\right|.
\end{align*}
Let
\begin{equation}\label{eq:w_n}
w_n \defeq  {\eta_n \over \sqrt{n}}\E\left[\sup_{0 \le u \le nt} 
{1 \over \sqrt{n}}\left|D_i(u)-u\right|\right],
\end{equation}
and
\[z_n \defeq  {\eta_n^2 \over n}\E\left[\left(\sup_{0 \le u \le nt} 
{1 \over \sqrt{n}}\left|D_i(u)-u\right|\right)^2\right].\]
Then we have, for all $i \ge 1$
\[\E\left[\vn{N_i^n}_t\right] \le w_n \quad \text{ and } \quad 
\E\left[\vn{N_i^n}_t^2\right] \le z_n.\]
By the FCLT, we have 
\[\sup_{0 \le u \le nt}{1 \over \sqrt{n}}\left|D_i(u) - u\right| \Rightarrow \sup_{0 \le u
\le t} |B(u)|,\]
where $B$ is a standard Brownian motion. We now have
\begin{align*}
 \E\left[\vn{N_i^n}_t \1\{A\}\right]  &\le \sqrt{z_n}
\sqrt{{\alpha  \over 2\log_\rho \eta_n} 
+ {1 \over \left(\log_\rho \eta_n\right)^2}}.
\end{align*}
Because $\eta_n = \bo{\sqrt{n}}$, as $n \to \infty$ we have
\[{\eta_n^2 \over n \log_\rho \eta_n} \to 0,\]
and thus
\begin{align*}
 \E\left[\vn{N_i^n}_t \1\{A\}\right]  &\le 
\sqrt{{z_n \alpha  \over 2\log_\rho \eta_n} 
+ {z_n \over \left(\log_\rho \eta_n\right)^2}} \to 0.
\end{align*}

Returning to \eqref{eq:breakdown}, we obtain
\begin{align*}
    \E\left[\vn{N_i^n}_t\right] &=
\E\left[\vn{N_i^n}_t\1\{A\}\right] +
\E\left[\vn{N_i^n}_t \1\{A^c\}\right] \\
&\le \sqrt{{z_n \alpha  \over 2\log_\rho \eta_n} 
+ {z_n \over \left(\log_\rho \eta_n\right)^2}} + \gamma_n,
\end{align*}
and thus
\[\lim_{n \to \infty} \max_{1 \le i \le i^*_n}
\E\left[\vn{N_i^n}_t\right]
 \le \lim_{n \to \infty} \sqrt{{z_n \alpha  \over 2\log_\rho \eta_n} 
+ {z_n \over \left(\log_\rho \eta_n\right)^2}} + \gamma_n = 0,\]
which establishes the claim \eqref{eq:kappa-exp-N-short}.

Now consider
\begin{align*}
\E\left[\vn{N^n}_{\rho,t}\right] &= \sum_{i \ge
1}\rho^{-i}\E\left[\vn{N_i^n}_t\right] \\
&\le \max_{1 \le i \le
i^*_n}\E\left[\vn{N_i^n}_t\right]\sum_{1
\le i \le i^*_n}\rho^{-i}
+ \sum_{i > i^*_n}\rho^{-i}\E\left[\vn{N_i^n}_t\right] \\
&\le {\rho \over \rho - 1}\max_{1 \le i \le
i^*_n}\E\left[\vn{N_i^n}_t\right]  + \sum_{i > i^*_n}\rho^{-i}w_n \\
&= {\rho \over \rho -1}\max_{1 \le i \le
i^*_n}\E\left[\vn{N_i^n}_t\right]  + {\rho^{-i
^*_n} w_n \over \rho -1}.
\end{align*}
Recall the definition of $w_n$ given by \eqref{eq:w_n} and observe that
\[{\eta_n \over \sqrt{n}}\rho^{-i^*} 
= {\eta_n^{1-\alpha/2} \over \sqrt{n}} \to 0,\]
as $n \to \infty$ implies the second term goes to zero. 
Recalling \eqref{eq:kappa-exp-N-short}, we conclude
\[\E\left[\vn{N^n}_{\rho,t}\right] \to 0.\]
This implies the convergence $N^n \to 0$ in $D^{\infty,\rho}_t$ 
equipped with the product topology.

The argument to show $\E\left[\vn{M^n}_{\rho,t}\right] \to 0$
is similar: we redefine $\tau^n_i$ as
\[\tau_i^n = \lambda_n d \int_0^t\left(T_{i}^n(s) - T_{i-1}^n(s) +
g^n(T_i^n(s)) - g^n(T_{i-1}^n(s))\right)ds.\]
By Lemma \ref{thm:g^m-lipschitz}, $g^{\eta_n}$ restricted to 
$D^{\eta_n+2}_t$ is Lipschitz continuous with constant $4^d$ and that
$\lambda_n \uparrow 1$, so we have the bound
\begin{align*}
\tau^i_n &\le td(1+4^d)\vn{T_i^n}_t + td4^d\vn{T_{i-1}^n}_t.
\end{align*}
This bound replaces \eqref{eq:tau-N-bound} and the rest of the argument
proceeds essentially identically to the $N^n$ case.
\end{proof}
We are now prepared to prove our main result:

\begin{proof}[Proof of Theorem \ref{thm:const-result}]
We first claim $\hat{T}^n \Rightarrow T$.
We established in Section \ref{sec:martingales} that 
$\hat{T}^n = f(T^n(0), -M^n+N^n, \lambda_n, \eta_n)$ where $f$ is the function
defined in Theorem \ref{thm:const-integral}. 
In Proposition \ref{thm:martingale-limit} we established $-M^n + N^n
\Rightarrow 0$. 
We prove $f(T^n(0), -M^n+N^n, \lambda_n, \eta_n) \Rightarrow f(T(0),0,1,\infty)$ 
directly. To do that, we choose a closed set $F \subset D_t^{\infty,\rho}$ and 
show that
\begin{align}\label{eq:convergence-claim}
\limsup_n \PP&\left(f(T^n(0), -M^n+N^n, \lambda_n, \eta_n) \in F\right) \notag \\
&\le \PP\left(f(T(0),0,1,\infty) \in F\right).
\end{align}
Our approach is to choose some large constant $K > 0$ and consider two cases.
Let $A_{K,n}$ be the event 
$\left\{\max\left(\vn{T^n(0)}_{\rho},\vn{-M^n+N^n}_{\rho,t}\right) > K\right\}$ 
and observe that assumption \eqref{eq:starting-expectation-bounded} and Proposition
\ref{thm:martingale-limit} imply
\[\limsup_n\PP\left(A_{K,n}\right) \defeq \epsilon_K \to 0 \quad \text{ as } \quad K \to \infty.\]
Let $Y^n = (T^n(0), -M^n+N^n, \lambda_n, \eta_n)$ and $Y = (T(0),0,1,\infty)$. Then
\begin{align*}
\limsup_n \PP(f(Y^n) \in F) &= \limsup_n \bigg(\PP\left(f(Y^n) \in F, A_{K,n}\right) \\
&\quad\quad\quad\quad +
\PP\left(f(Y^n) \in F, A_{K,n}^c\right)\bigg) \\
&\le \epsilon_K + \limsup_n\PP\left(f(Y^n) \in F, A_{K,n}^c\right).
\end{align*}
Recall \eqref{eq:martingale-arrival}-\eqref{eq:martingale-departure} and observe that
$M^n(0) = N^n(0) = 0$. Thus $A_{K,n}^c$ and assumption \eqref{eq:start-inside} imply 
$Y^n \in \region{K}$ for sufficiently large $n$.
For such $n$, we have
\begin{align*}
\PP\left(f(Y^n) \in F, A_{K,n}^c\right) 
&= \PP\left(Y^n \in f^{-1}(F), A_{K,n}^c \right) \\
&= \PP\left(Y^n \in f^{-1}(F) \cap \region{K}\right).
\end{align*}
Because both $F$ and $\region{K}$ are closed and $f$ is continuous on $\region{K}$, 
$ f^{-1}(F) \cap \region{K}$ is closed.
Thus the convergence $Y^n \Rightarrow x$ implies
\begin{align*}
\limsup_n\PP\left(Y^n \in f^{-1}(F) \cap \region{K} \right) 
&\le \PP\left(Y \in f^{-1}(F) \cap \region{K} \right) \\
&\le \PP\left(Y \in f^{-1}(F) \right) \\
&= \PP\left(f(Y) \in F \right).
\end{align*}
Thus
\[\limsup_n\PP\left(f(Y^n) \in F, A_{K,n}^c\right) 
\le \PP\left(f(Y) \in F\right),\]
and
\[\limsup_n \PP\left(f(Y^n) \in F\right) \le \epsilon_K + 
\PP\left(f(Y) \in F\right).\]
Taking the limit $K \to \infty$ on both sides establishes \eqref{eq:convergence-claim}.
Thus
\[\hat{T}^n \Rightarrow f(T(0),0,1,\infty).\]
By definition, we have $f(T(0),0,1,\infty) = T$. Thus we conclude
\[\hat{T}^n \Rightarrow T.\]

By construction, $\hat{T}^n(t)$ and $T^n(t)$ are identical for $t
\in [0,t_n^*]$. Thus it remains to show that for all $t \ge 0$, $\PP(t_n^*
\le t) \to 0$ as $n \to \infty$. Let 
\[p_n = \PP(t_n^* \le t) = 
\PP\left(\exists \text{ } i \le i^*_n \st \vn{T_i^n}_t \ge
\eta_n^{\alpha}\right).\]
We have
\begin{align}
\E\left[\vn{T^n}_{\rho,t}\right] &\ge \sum_{1 \le i \le i^*_n} \rho^{-i} \E\vn{T_i^n}_t \notag \\
&\ge \rho^{-i^*_n}\sum_{1 \le i \le i^*_n}\E\vn{T_i^n}_t \notag \\
&\ge \rho^{-i^*_n}\eta_n^\alpha p_n \notag\\
&= \rho^{-{\alpha \over 2}\log_\rho \eta_n }\eta_n^{\alpha} p_n \notag\\
&= \eta_n^{\alpha / 2} p_n, \label{eq:exp-lower-bound}
\end{align}
We have $\eta_n^{\alpha / 2} \to \infty$ as $n \to \infty$
By Lemma \ref{thm:const-expectation-finite} with $\gamma = \rho$ and assumption \ref{eq:starting-expectation-bounded}, 
we have $\lim_{n \to \infty}\E\vn{T^n}_{\rho,t} < \infty$, so \eqref{eq:exp-lower-bound} does not diverge to infinity,
which implies $p_n \to 0$. This implies $T^n \Rightarrow T$, as desired.
\end{proof}

\section{Open questions.}\label{sec:Conclusions}
In addition to the questions and conjectures already proposed in Section
\ref{sec:steady-state},
another potential future direction would be to characterize the detailed
behavior of
queues of length $\log_d \eta_n + \bo{1}$. 
Not only do we conjecture that in steady state
almost all queues are of this type, but we also expect that over finite time
the process tracking such queues converges to a
diffusion process rather than a deterministic system. The methods used in this
paper do not naturally translate to the ``intermediate length queue'' regime.

\bibliographystyle{plain}
\bibliography{supermarket}

\end{document}